\documentclass[12pt,a4paper]{amsart}
\usepackage[latin2]{inputenc}
\usepackage[mathscr]{eucal}
\usepackage{enumerate}

\usepackage{graphicx}
\usepackage{amsmath}
\usepackage{amsthm}
\usepackage{amssymb}
\usepackage{amsfonts}
\usepackage{latexsym}
\usepackage{color}
\usepackage[active]{srcltx}

\newcommand{\comment}[1]{}

\newcommand{\Norm}[1]{{\left\|{#1}\right\|}}

\newcommand{\PP}{{\mathcal P}}

\newcommand{\PK}{{\mathcal P}_n(K)}
\newcommand{\ze}{\zeta}
\newcommand{\DK}{\partial K}

\newcommand{\sj}{{\frac{\sin\varphi_j}{r_j}}}

\newcommand{\LL}{{\mathcal L}}

\newcommand{\Lj}{{I}_j}

\newcommand{\gs}{\gamma^{*}}
\newcommand{\as}{\alpha^{*}}
\newcommand{\gw}{\widetilde{\gamma}}
\newcommand{\gws}{\widetilde{\gamma^{*}}}
\newcommand{\aw}{\widetilde{\alpha}}
\newcommand{\aws}{\widetilde{\alpha^{*}}}

\newcommand{\HH}{{\mathcal H}}
\newcommand{\GG}{{\mathcal G}}

\newcommand{\arc}[1]{\widetilde{#1}}

\newcommand{\A}{{\mathcal A}}
\newcommand{\AM}{{\mathcal A}_{-}}
\newcommand{\AP}{{\mathcal A}_{+}}

\newcommand{\EE}{{\mathcal E}}
\newcommand{\FF}{{\mathcal F}}
\newcommand{\I}{{\mathcal I}}

\newcommand{\Z}{{\mathcal Z}}

\newcommand{\W}{{\mathcal W}}
\newcommand{\RR}{{\mathbb R}}
\newcommand{\CC}{{\mathbb C}}
\newcommand{\ZZ}{{\mathbb Z}}
\newcommand{\NN}{{\mathbb N}}
\newcommand{\TT}{{\mathbb T}}
\newcommand{\DD}{{\mathbb D}}
\newcommand{\II}{{\mathbb I}}

\newcommand{\Inner}[2]{{\left\langle #1, #2 \right\rangle}}

\newcommand{\Var}{{\rm Var}}

\newcommand{\dist}{{\rm dist\,}}
\newcommand{\intt}{{\rm int\,}}

\newcommand{\con}{{\rm con\,}}
\newcommand{\diam}{{\rm diam\,}}
\newcommand{\width}{{\rm width\,}}
\newcommand{\cl}{{\rm cl\,}}

\newcommand{\bnu}{\boldsymbol{\nu}}

\newtheorem{theorem}{Theorem}
\newtheorem{theoremone}{Theorem}

\newtheorem{corollary}{Corollary}
\newtheorem{lemma}{Lemma}
\newtheorem{claim}{Claim}

\newtheorem{definition}{Definition}
\newtheorem{conjecture}{Conjecture}
\theoremstyle{definition}

\theoremstyle{definition}

\newtheorem{remark}{Remark}

\newcommand{\de}{\delta}

\newcommand{\al}{\alpha}
\newcommand{\ff}{\varphi}

\newcounter{othm}
\def\theothm{\Alph{othm}}
\newenvironment{othm}{
  \em
  \vskip 0.10in
  \refstepcounter{othm}
  \noindent{\bf Theorem\ \theothm}
}{\vskip 0.10in}

\newenvironment{olemma}{
  \em
  \vskip 0.10in
  \refstepcounter{othm}
  \noindent{\bf Lemma\ \theothm}
}{\vskip 0.10in}

\reversemarginpar 

\begin{document}

\title[$L^q$ Tur\'an inequalities on convex domains]{Tur\'an-Er\H od type converse Markov inequalities \\
on general convex domains of the plane in $\boldsymbol{L^q}$}

\author{Polina Yu. Glazyrina, Szil\'ard Gy. R\'ev\'esz}

\address
{Polina Yu. Glazyrina \newline  \indent  Institute of Natural Sciences and Mathematics, Ural Federal University, \newline  \indent Ekaterinburg, RUSSIA \newline  \indent  and \newline  \indent
Institute of Mathematics and Mechanics, \newline  \indent Ural Branch of the Russian Academy of Sciences, \newline  \indent Ekaterinburg, RUSSIA}
\email{polina.glazyrina@urfu.ru}

\address{Szil\'ard Gy. R\'ev\'esz
\indent A. R\'enyi Institute of Mathematics \newline \indent Hungarian
Academy of Sciences, \newline \indent Budapest, Re\'altanoda utca
13-15. \newline \indent 1053 HUNGARY} \email{revesz.szilard@renyi.mta.hu}

\date{}

\begin{abstract} In 1939 P. Tur\'an started to derive lower estimations on
the norm of the derivatives of polynomials of (maximum) norm 1 on
$\II:=[-1,1]$ (interval) and $\DD:=\{z\in\CC~:~|z|\le 1\}$ (disk),
under the normalization condition that the zeroes of the polynomial in
question all lie in $\II$ or $\DD$, respectively. For the maximum norm
he found that with $n:=\deg p$ tending to infinity, the precise growth
order of the minimal possible derivative norm is $\sqrt{n}$ for $\II$
and $n$ for $\DD$.

J. Er\H od continued the work of Tur\'an considering other domains.
Finally, a decade ago the growth of the minimal possible $\infty$-norm
of the derivative was proved to be of order $n$ for all compact convex
domains.

Although Tur\'an himself gave comments about the above oscillation
question in $L^q$ norms, till recently results were known only for
$\DD$ and $\II$. Recently, we have found order $n$ lower estimations
for several general classes of compact convex domains, and conjectured
that even for arbitrary convex domains the growth order of this quantity
should be $n$. Now we prove that in $L^q$
norm the oscillation order is at least $n/\log n$ for all compact
convex domains.
\end{abstract}

\maketitle

\centerline{
\emph{Dedicated to Sergey V. Konyagin on the occasion
of his sixtieth birthday}}

\let\oldfootnote\thefootnote
\def\thefootnote{}
\footnotetext{} \footnotetext{This work was supported by the Russian
Foundation for Basic Research (Project No. 15-01-02705) and by the
Program for State Support of Leading Universities of the Russian
Federation (Agreement No. 02.A03.21.0006 of August 27, 2013) and by
Hungarian National Research, Development and Innovation Funds \#'s K-109789, K-119528.}
\let\thefootnote\oldfootnote



\tableofcontents

%%%%%%%%%%%%%%%%%%%%%%%%%%%%%%%%%%%%%%%%%%%%%
%%%%%%%%%%         Introduction
%%%%%%%%%%%%%%%%%%%%%%%%%%%%%%%%%%%%%%%%%%%%%

\section{Introduction}\label{sec:intro}

Denote by $K\Subset \CC$ a compact subset of the complex plane,
with the most notable particular cases being the unit disk
$\DD:=\{ z \in \CC ~:~ |z|\le 1\}$ and the unit interval
$\II:=[-1,1]$.

As a kind of converse to the classical inequalities of
Bernstein \cite{Bernstein, Bernstein_com, Riesz} and Markov
\cite{Markov} on the upper estimation of the norm of the
derivative of polynomials, in 1939 Paul Tur\'an \cite{Tur}
started to study converse inequalities of the form
$\Norm{p'}_K\ge c_K n^A \Norm{p}_K$. Clearly such a converse
can only hold if further restrictions are imposed on the
occurring polynomials $p$. Tur\'an assumed that all zeroes of
the polynomials belong to $K$. So denote the set of complex
(algebraic) polynomials of degree (exactly) $n$ as $\PP_n$, and
the subset with all the $n$ (complex) roots in some set
$K\subset\CC$ by $\PK$.

Denote by $\Gamma$ the boundary of $K$. The (normalized)
quantity under our study in the present paper is the ``inverse
Markov factor" or "oscillation factor"
\begin{equation}\label{Mdef}
M_{n,q}:=M_{n,q}(K):=\inf_{p\in \PK} M_q(p) \qquad \text{\rm with} \qquad
M_q(p):=\frac{\Norm{p'}_{L^q(\Gamma)}}{\Norm{p}_{L^q(\Gamma)}},
\end{equation}
where, as usual,
\begin{align}\label{Oscillationdef}
\Norm{p}_{q}:&=\Norm{p}_{L^q(\Gamma)}:=\left(\int_{\Gamma} |p(z)|^q|dz|\right)^{1/q},
\quad (0<q<\infty) \notag
\\ \Norm{p}_K:=\|p\|_\infty:&=\Norm{p}_{L^\infty(\Gamma)}=\Norm{p}_{L^\infty(K)}=\sup_{z\in \Gamma} |p(z)|=\sup_{z\in K}|p(z)|.
\end{align}
Note that for $0<q<\infty$ the $L^q(\Gamma)$ norm remains finite if $\Gamma$ is a rectifiable curve.

\begin{othm}{\bf(Tur\'an).}\label{oth:Turandisk}
If $p\in \PP_n(\DD)$, where $\DD$ is the unit disk, then we
have
\begin{equation}\label{Turandisk}
\Norm{p'}_\DD\ge \frac n2 \Norm{p}_\DD~.
\end{equation}
If $p\in\PP_n(\II)$, then we have
\begin{equation}\label{Turanint}
\Norm{p'}_\II\ge \frac {\sqrt{n}}{6} \Norm{p}_\II~.
\end{equation}
\end{othm}

Inequality \eqref{Turandisk} of Theorem \ref{oth:Turandisk} is
best possible. Regarding \eqref{Turanint}, Tur\'an pointed out
by example of $(1-x^2)^{n}$ that the $\sqrt{n}$ order cannot be
improved upon, even if the constant is not sharp, see also
\cite{BabenkoMN86, LP}. The precise value of the constants and
the extremal polynomials were computed for all fixed $n$ by
Er\H{o}d in \cite{Er}.

We are discussing Tur\'an-type inequalities~\eqref{Mdef} for
general convex sets, so some geometric parameters of the compact convex
domain $K$ are involved naturally.
We write $d:=d_K:=\diam (K)$ for the {\em diameter} of $K$,
and $w:=w_K:={\width}(K)$ for the {\em minimal width} of $K$.
That is,
\begin{align}\label{diameterdef}
d:=d_K & := \max_{z', z''\in K} |z'-z''|, \\ \notag
w:=w_K& := \min_{\gamma\in [-\pi,\pi]} \left( \max_{z\in K} \Re
(ze^{i\gamma}) - \min_{z\in K} \Re (ze^{i\gamma}) \right).
\end{align}
Note that a compact convex \emph{domain} is a closed,
bounded, convex set $K\subset\CC$ \emph{with nonempty
interior}, hence $0<w_K\le d_K<\infty$.

The key to \eqref{Turandisk} is the following straightforward
observation.
\begin{olemma}{\bf(Tur\'an).}\label{Tlemma} Assume that $z\in\partial K$ and
that there exists a disc $D_R=\{\zeta\in \CC~:~ |\ze-z_0|\le R\}$ of
radius $R$ so that $z\in\partial D_R$ and $K\subset D_R$. Then for all
$p\in\PK$ we have
\begin{equation}\label{Rdisc}
|p'(z)| \ge \frac n{2R} |p(z)|.
\end{equation}
\end{olemma}

For the easy and direct proof see any of the references \cite{Tur, LP, Rev3, SofiaCAA, PR}. Levenberg and Poletsky \cite{LP} found it worthwhile to
formally define the crucial property of convex sets, used here.

\begin{definition}[{\bf Levenberg-Poletsky}]\label{def:Rcircular}
A set $K\Subset \CC$ is called \emph{$R$-circular}, if for any  $z\in\partial K$ there exists a disk $D_R$ of radius $R$, such that
$z\in\partial D_R$ and $D_R\supset K$ .
\end{definition}

Thus for any $R$-circular $K$ and $p\in \PP_n(K)$ at the
boundary point $z\in\partial K$ with $\|p\|_K=|p(z)|$ we can
draw the disk $D_R$ and get \eqref{Rdisc} to hold for $p \in \PK, ~ z \in \DK$.

Er\H od continued the work of Tur\'an already the same year,
investigating the inverse Markov factors of domains with some favorable
geometric properties. The most general domains with $M_{n,\infty}(K)\gg
n$, found by Er\H od, were described on p. 77 of \cite{Er}.

\begin{othm}{\bf(Er\H od).}\label{oth:transfquarter}
Let $K$ be any convex domain bounded by finitely many Jordan arcs, joining at vertices with angles $<\pi$, with all the arcs being $C^2$-smooth and being either straight lines of length $<\Delta(K)$, where $\Delta(K)$ stands for the transfinite diameter of $K$, or having positive curvature bounded away from $0$ by a fixed constant $\kappa>0$. Then there is a constant $c(K)$, such that $M_{n,\infty}(K)\geq c(K) n$ for all $n\in\NN$.
\end{othm}

As is discussed in \cite{PR}, this result covers the case of regular
$k$-gons for $k\ge 7$, but not the square, e.g.,
which was also proved to have order $n$ oscillation
 but only much later, by Erd\'elyi \cite{E}.

A lower estimate of the inverse Markov factor for all compact convex
sets (and of the same $\sqrt{n}$ order as was known for the interval)
was obtained in full generality by Levenberg and Poletsky, see
\cite[Theorem 3.2]{LP}.

Since $\sqrt{n}$ was already known to be the right order of growth for the inverse Markov factor of $\II$, it remained to clarify the right order of oscillation for compact convex \emph{domains} with nonempty interior. This was solved a decade ago in \cite{Rev2}.

\begin{othm}{\bf(Hal\'{a}sz--R\'ev\'esz).}\label{th:convexdomain}
Let $K\subset \CC$ be any compact convex domain. Then for all  $p\in
\PK$ we have
\begin{equation}\label{genrootineq}
\Norm{p'}_K\ge 0.0003 \frac{w_K}{d_K^2} n  \Norm{p}_K~.
\end{equation}
\end{othm}

For the fact that it is indeed the precise order -- moreover, $M_{n,\infty}(K)$ can only be within an absolute constant multiple of the above lower estimation -- see \cite{Rev3, PR, SofiaCAA}.

\bigskip

There are many papers dealing with the $L^q$-versions of Tur\'an's inequality for the disk $\DD$, the interval $\II$, or for the period (one dimensional torus or
circle) $\TT:=\RR/2\pi\ZZ$ (here with considering only real trigonometric polynomials). A nice review of the results obtained before 1994 is given in \cite[Ch. 6, 6.2.6, 6.3.1]{MMR}.

Already Tur\'an himself mentioned in \cite{Tur} that on the perimeter of the disk $\DD$ actually the pointwise inequality \eqref{Rdisc} holds \emph{at all points} of $\partial \DD$ or $\DK$. As a corollary, for any $q>0$, $\left(\int_{|z|=1}|p'(z)|^q|dz|\right)^{1/q} \ge  \frac n{2}\left( \int_{|z|=1}|p(z)|^q|dz|\right)^{1/q}$. Consequently, Tur\'an's result \eqref{Turandisk} extends to all weighted $L^q$-norms on the perimeter, including all $L^q(\partial \DD)$ norms.

The estimation of the $L^q$ norm, or of any norm including e.g. any weighted $L^q$-norms, goes the same way if we have a pointwise estimation for all, (or for linearly almost all), boundary points. This observation was explicitly utilized first in \cite{LP}.

In case we discuss maximum norms, one can assume that $|p(z)|$ is maximal, and it suffices to obtain a lower estimation of $|p'(z)|$ only at such a special point -- for general norms, however, this is not sufficient. The above results work only for we have a pointwise inequality of the same strength \emph{everywhere}, or almost everywhere. The situation becomes considerably more difficult, when such a statement cannot be proved. E.g. if the domain in question is not strictly convex, i.e. if there is a line segment on the boundary, then the zeroes of the polynomial can be arranged so that even some zeroes of the derivative lie on the boundary, and at such points $p'(z)$ -- even $p'(z)/p(z)$ -- can vanish. As a result, at such points no fixed lower estimation can be guaranteed, and lacking a uniformly valid pointwise comparision of $p'$ and $p$, a direct conclusion cannot be drawn either.

This explains why already the case of the interval $\II$ proved to be much more complicated for $L^q$ norms. This was solved by Zhou in a series of papers \cite{Zhou84, Zhou86, Zhou92, Zhou93, Zhou95}. For more discussions on these results, as well as related results on the interval, period and circle, see the detailed survey in \cite{PR} and the introduction of \cite{PR2}, as well as the original works of Babenko and Pichugov \cite{BabenkoU86}, Bojanov \cite{Bojanov93}, Varma~\cite{Varma88_83} Babenko et al. \cite{BabenkoU86, BabenkoMN86} Bojanov \cite{Bojanov96} and Tyrygin \cite{Tyrygin88, TyryginDAN88, KBL}; see also \cite{Varma79, UV96, WangZhou02, KBL}.

 \ The classical inequalities of Bernstein and Markov are generalized for various differential operators, too, see \cite{Arestov}.
In this context, also Tur\'an type converses have been already investigated e.g. by Akopyan~\cite{Akopyan00} and Dewan et al. \cite{Detal}

\bigskip

Involving the Blaschke Rolling Ball Theorem, and even recent extensions of it, certain classes of domains were proved to admit order $n$ oscillation factors in $L^q$, see \cite[Theorem 2]{PR2}. More importantly, however, combining these $R$-circular classes and the most general classes considered by Er\H od in Theorem \ref{oth:transfquarter} (for $\|\cdot\|_\infty$), we could obtain the next result, see \cite[Theorem 1]{PR2}.

\begin{othm}{\bf (Glazyrina--R\'ev\'esz).}\label{th:ErodType}
Let $K\Subset \CC$ be an $E(d,\Delta,\kappa,\xi,\delta)$-domain. Then
for any $q\ge 1$ there exists a constant $c=c_K$ (depending explicitly
on the parameters $q,d,\Delta,\kappa,\xi,\delta$) such that for all
$n\in \NN$ and $p\in\PK$ we have $\|p'\|_q \ge c_K n \|p\|_q$.
\end{othm}

Here the definition of a ``generalized Er\H od type domain''
$E(d,\Delta,\kappa,\xi,\delta)$ is basically the one used in Theorem
\ref{oth:transfquarter}, but with skipping the assumption of $C^2$
smoothness and relaxing the $\ddot{\gamma} \geq \kappa$ everywhere
assumptions on the curved pieces of the boundary: here $\ddot{\gamma}
\geq \kappa$ is assumed only (linearly) almost everywhere.

More discussion of this definition would us lead aside from our main
line of progress, so we direct the reader for more details and
explanations (as well as for the proof) to the original paper
\cite{PR2}.

\bigskip

Recently, we obtained some order $n$ oscillation results for certain
further convex domains without any condition on the curvature. To
formulate this, let us first recall another geometrical notion, namely,
the
\emph{depth} of a convex domain $K$ as
\begin{align}\label{eq:bodydepth}
h_K &:=\sup \{ h\ge 0 ~:~ \forall \zeta\in\partial K ~\exists ~ {\rm a}~{\rm normal} ~ {\rm line} ~ \ell ~ {\rm at} ~ \zeta ~ {\rm to} ~ K ~ {\rm with} ~~ |\ell\cap K| \ge h \}.
\end{align}
We say that the convex domain $K$ has {\em fixed depth} or {\em
positive depth}, if $h_K>0$. The class of convex domains having
positive depth contains all smooth compact convex domains, and also all
polygonal domains with no vertex with an acute angle. However, observe
that the regular triangle has $h_K=0$, as well as any polygon having
some acute angle. For more about this class see \cite{PR}, where also
the following was proved.
\begin{othm}{\bf (Glazyrina--R\'ev\'esz).}\label{th:posdepth} Assume that $K\Subset \CC$ is a compact convex domain having positive depth $h_K>0$. Then for any $q\ge 1, n\in \NN$ and $p\in\PK$ it holds 
\begin{equation}\label{eq:ordernLq}
\|p'\|_{q} \ge c_K n  \|p\|_{q} \qquad \left( c_K:=\frac{h_K^4}{3000 d_K^5} \right).
\end{equation}
\end{othm}

From the other direction, we also proved that one cannot expect more than order $n$ growth
 of $M_{n,q} (K)$. In fact, in this direction our result was more general, but here we recall only a
combination of Theorem 5 and Remark 6 of \cite{PR}.
\begin{othm}{\bf (Glazyrina--R\'ev\'esz).}\label{th:orderupper} Let $K\Subset \CC$
 be any compact, convex domain.
Then for any $q \ge 1$ and any $n \in\NN$ there exists a polynomial  $p\in \PK$
 satisfying $\|p'\|_q < \dfrac{15}{d_K} n \|p\|_q $.
\end{othm}

In \cite{PR} we formulated the following conjecture, too.
\begin{conjecture}\label{conj:cn} For all compact convex domains $K\Subset \CC$ there exist $c_K>0$ such that for any $p\in\PK$ we have $\|p'\|_{L^q(\DK)} \ge c_K n \|p\|_{L^q(\DK)}$. That is, for any compact convex domain $K$ the growth order of $M_{n,q}(K)$ is precisely $n$.
\end{conjecture}
Also we pointed out that in the positive (Tur\'an--Er\H od type
oscillation) direction, apart from the above findings for various
classes, no completely general result is known, not even with a lower
estimation of any weaker order than conjectured. This situation was
compared to the situation in the development of the $\infty$-norm case,
where a general lower estimation result, valid for all compact convex
domains, was first proved only in 2002.

The aim of the present work is to prove the validity of a general lower estimation.
\begin{theorem}\label{th:nlogn} Let
$K\Subset \CC$ be any compact convex domain and $q\ge 1$.
Then there exists a constant $c_K$ such that for $n\ge n_0(q,K)$  and all $p\in \PK$  we have
\begin{equation}\label{genrootineq}
\Norm{p'}_{q}\ge c_K \frac{n}{\log n}  \Norm{p}_{q}~.
\end{equation}
In other words, for compact convex domains we always have $c_K\dfrac{n}{\log n} \le M_{n,q} \le C_K n$.
\end{theorem}
Note that this, although indeed falling short of Conjecture \ref{conj:cn}, clearly exceeds the order $\sqrt{n}$, known for the interval $\II$.

\section{Some basic geometrical notations and facts}\label{sec:geom}

We need to fix geometrical notations. Let us start with a \emph{convex, compact domain} $K\Subset \CC$. Then its interior $\intt K \ne \emptyset$ and $K=\overline{\intt K}$, while its boundary $\Gamma:=\partial K$ is a convex Jordan curve. More precisely, $\Gamma = {\mathcal R}(\gamma)$ is the \emph{range} of a continuous, convex, closed Jordan curve $\gamma$ on the complex plane $\CC$.

If the parameter interval of the Jordan curve $\gamma$ is $[0,L]$, then this means, that $\gamma: [0,L] \to \CC$ is continuous, convex, and one-to-one on $[0,L)$, while $\gamma(L)=\gamma(0)$. While this compact interval parametrization is the most used setup for curves, we need the essentially equivalent interpretations with this, too: one is the definition over the torus $\TT:=\RR/L\ZZ$ and the other is the periodically extended interpretation with $\gamma(t):=\gamma(t-[t/L]L)$ defined periodically all over $\RR$. If we need to distinguish, we will say that $\gamma:\RR\to\CC$ and $\gs : \TT:=\RR/L\ZZ \to \CC$, or equivalently, $\gs : [0,L] \to \CC$ with $\gs(L)=\gs(0)$.

As the curves are convex, they always have finite arc length $L:=|\gs|$. Accordingly, we will restrict ourselves to \emph{parametrization with respect to arc length}. The parametrization $\gamma: \RR \to \partial K$ defines a unique ordering of points, which we assume to be positive in the counterclockwise direction, as usual. When considered locally, i.e. with parameters not extending over a set longer than the period, this can be interpreted as ordering of the image (boundary) points themselves: we always implicitly assume, that a proper cut of the torus $\TT$ is applied at a point to where the consideration is not extended, and then for the part of boundary we consider, the parametrization is one-to-one and carries over the ordering of the cut interval to the boundary.

Arc length parametrization has an immediate consequence also regarding the derivative, which must then have $|\dot{\gamma}|=1$, whenever it exists, i.e. (linearly) a.e. on $[0,L)\sim \TT$. Since $\dot{\gamma} :\RR \to \partial \DD$, we can as well describe the value by its angle or argument: the derivative angle function will be denoted by $\alpha:=\arg \dot{\gamma} : \RR \to \RR$. Since, however, the argument cannot be defined on the unit circle without a jump, we decide to fix one value and then define the extension continuously: this way $\alpha$ will not be periodic, but we will have rotational angles depending on the number of (positive or negative) revolutions, if started from the given point. With this interpretation, $\alpha$ is an a.e. defined nondecreasing real function with $\alpha(t)-\frac{2\pi}{L} t$ periodic (by $L$) and bounded. By convexity, angular values attained by $\alpha(t)$ are then ordered the same way as boundary points and parameters. In particular, for a subset not extending to a full revolution, the angular values are uniquely attached to the boundary points and parameter values and they are ordered the same way by considering a proper cut.

With the usual left- and right limits $\alpha_{-}$ and $\alpha_{+}$ are the left- resp. right-continuous extensions of $\alpha$. The geometrical meaning is that if for a parameter value $\tau$ the corresponding boundary point is $\gamma(\tau)=\ze$, then $[\alpha_{-}(\tau),\alpha_{+}(\tau)]$ is precisely the interval of values $\beta \in \TT$ such that the straight lines $\{\zeta+e^{i\beta}s~:~ s\in \RR\}$ are supporting lines to $K$ at $\zeta \in \partial K$. We will also talk about half-tangents: the left- resp. right- half-tangents are the half-lines emanating from $\zeta$ and progressing towards $-e^{i\alpha_{-}(\tau)}$ or $e^{i\alpha_{+}(\tau)}$, resp. The union of the half-lines $\{\zeta+e^{i\beta}s~:~ s\ge 0\}$ for all $\beta\in [\alpha_{+}(\tau),\pi-\alpha_{-}(\tau)]$ is precisely the smallest cone with vertex at $\zeta$ and containing $K$.

We will interpret $\alpha$ as a multi-valued function, assuming all the values in $[\alpha_{-}(\tau),\alpha_{+}(\tau)]$ at the point $\tau$.
Restricting to the periodic (finite interval) interpretation of $\gs: [0,L)\to \CC$, without loss of generality we we may assume that $\as:=\arg(\dot{\gs}) :[0,L]\to [0,2\pi]$. In this regard, we can say that $\as:\RR/L\ZZ \to \TT$ is of bounded variation, with total variation (i.e. total increase) $2\pi$--the same holds for $\alpha:\RR\to\RR$ over one period.

The curve $\gamma$ is differentiable at $\zeta=\gamma(\theta)$ if and only if $\alpha_{-}(\theta)=\alpha_{+}(\theta)$; in this case the unique tangent of $\gamma$ at $\zeta$ is $\zeta+e^{i\alpha}\RR$ with $\alpha=\alpha_{-}(\theta)=\alpha_{+}(\theta)$.

It is clear that interpreting $\alpha$ as a function on the boundary points $\zeta\in \partial K$, we obtain a parametrization-independent function: to be fully precise, we would have to talk about $\gw$, $\gws$, $\aw$ and $\aws$. In line with the above, we consider $\aw$, resp $\aws$ \emph{multivalued functions}, all admissible supporting line directions belonging to $[\alpha_{-}(\tau),\alpha_{+}(\tau)]$ at $\zeta=\gamma(\tau)\in \DK$ being considered as $\aw$-function values at $\ze$. At points of discontinuity $\alpha_{\pm}$ or $\as_{\pm}$ and similarly $\aw_{\pm}$ resp. $\aws_{\pm}$ are the left-, or right continuous extensions of the same functions.

A convex domain $K$ is called {\em smooth}, if it has a unique
supporting line at each boundary point of $K$. This occurs if and only if
$\alpha_{\pm}:=\alpha$ is continuously defined for all values of
the parameter. For obvious geometric reasons we call the jump function  $\Omega:=\alpha_{+}-\alpha_{-}$ the {\em supplementary angle} function. This is identically zero almost everywhere (and in fact except for a countable set), and has positive values such that the total sum of the (possibly infinite number of) jumps does not exceed the total variation of $\alpha$, i.e. $2\pi$.

For a supporting line $\zeta+e^{i\beta}\RR$ at the boundary point $\ze\in\DK$ and oriented positively (so that $K$ lies in the halfplane $\{z\in\CC~:~ \beta \le \arg (z-\ze) \le \beta+\pi\}$)
the corresponding (outer) normal vector is ${\bnu}(\zeta):= e^{i(\beta-\pi/2)}$.

The family of all the (outer) normal vectors consists precisely of the vectors satisfying
$\Inner{z-\zeta}{\bnu} \le 0$ ($\forall z \in K$) with the usual $\RR^2$ scalar product, or equivalently, $\Re \left((z-\ze)\overline{\bnu}\right))\le 0$ (where $\overline{\bnu}$ is just the conjugate of the complex number $\bnu$).

Here we introduce a few additional notations, too.
First, we will write $\de(\ze,\varphi):=\de_K(\ze,\varphi):=|K\cap(\ze+e^{i\varphi} \RR)|$. Further, to denote the ``opposite endpoint"
 of the intersection line segment we will use the notation
$$D:=D(\ze):=D(\ze,\varphi):=D_K(\ze,\varphi),
$$ so that $K\cap (\ze+e^{i\varphi} \RR) =[\ze,D(\ze)]$ -- of course, in particular cases even $D(\ze)=\ze$ and $\de(\ze,\ff)=0$ is possible.

The following easy, but useful observation will be used several times in various situations.

\begin{claim}\label{cl:triangle}
Let $\ze \ne \ze' \in \DK$ and assume that $t=\ze+e^{i\ff}\RR_{+}$, $t'=\ze'+e^{i\ff'}\RR_{+}$ are two halflines, emanating from $\ze$ and $\ze'$, respectively, and having the (subderivative or half-tangent) property that $t\cap \intt K=\emptyset$ and also $t'\cap \intt K=\emptyset$. Assume that these halflines intersect in a point $T:=t\cap t'$. Write $\ell$ for the straight line connecting $\ze$ and $\ze'$, and assume that neither $t$, nor $t'$ is included (so is not parallel to) $\ell$, so that $T$ is in one of the open halfplanes of $\CC\setminus \ell$; denote this halfplane by $H$. Finally, put $\triangle:=\triangle_{\ze,T,\ze'}:=\con (\ze, T, \ze')$ for the triangle with vertices $\ze, T, \ze'$.

Then we have that $(H \cap K) \subset \triangle$.
\end{claim}
\begin{proof} Assume, as we may, that $\ze=-i$, $\ze'=i$, whence $\ell$ is the imaginary axis, and that $H=\{\Re z >0\}$ is the right halfplane, say. This means that both halflines $t$ and $t'$ are contained in $H$, cutting $H$ into four convex components, all bounded by (parts of the) straight lines $\ell, t, t'$: number them as $H_1,\dots,H_4$. One is esentially the triangle $\triangle$ but beware of the boundary: in precise terms, $H_1=\triangle\setminus[\ze,\ze']$, as the side $[\ze,\ze']$ of $\triangle$ falls on $\ell$, not contained in the open halfplane $H$. Also there are three other unbounded ones $H_2, H_3, H_4$.

The only component, which has both points $\ze, \ze'$ in its boundary $\partial H_j$, is necessarily the one with $\ze":=(\ze+\ze')/2=0$ in its boundary: this is $H_1$. Note that there exists a small $r>0$ with the property that $\{ z=\rho e^{i\ff}~:~ -\pi/2<\ff<\pi/2, \ 0<\rho<r\} \subset H_1$. Also, $0 \in K$ by convexity of $K$.

If $\intt K \cap H= \emptyset$ then also $K\cap H =\emptyset$ because $H$ is an \emph{open} halfplane and $K$ is \emph{fat}, \cite[Corollary 2.3.9]{Webster} i.e. all its (interior or boundary) points are limits of interior points. So in this case there remains nothing to prove.

So let us consider the case when $\intt K \cap H \ne \emptyset$. As $H$ is open, $(\intt K ) \cap H = \intt (K \cap H)$. Now we want to prove that then $(\intt K \cap H) \subset \triangle$.

Once we prove this, it will suffice, as for $K\cap H$ being a convex domain with nonempty interior it is also \emph{fat}, and thus $K\cap H \subset \cl(\intt K \cap H) \subset \cl(\triangle)=\triangle$, as needed.

So take any point $Z\in (\intt K \cap H)$ and assume for contradiction that $Z\not\in \triangle$.

Let now $z:=\rho e^{i\arg(Z)}=\rho Z/|Z|$ with some $\rho <r$: then $z\in \triangle \cap H$. As $0 \in K$, we will have $(0,Z] \subset \intt K$ in view of convexity of $K$; so in particular $[z,Z] \subset \intt K$.

As $z \in \triangle$ and $Z \not \in \triangle$, there exists a boundary point $B\in \partial \triangle$ on the segment $[z,Z]$: $B\in (\partial \triangle \cap [z,Z])$. So, $B \in (\partial \triangle \cap \intt K)=(\partial H_1 \cap \intt K)$. But it is also within $H$, where the boundary line segments of any component of $H$ can consist of only pieces of $t \cup t'$, free of $\intt K$ by assumption -- which is a contradiction.
\end{proof}

There are obvious, yet important consequences of the above, which we will use throughout our reasoning. First, if $\ze, \ze' \in \DK$ are two boundary points with $s:=|\ze-\ze'|<w$, then the tangent lines at these points cannot be distinct and parallel (as $K$ is not contained in any strip less wide than $w_K$). So, if $t\ne \ell$ and $t'\ne \ell$ is also assumed, then taking appropriate half-lines of these tangents, there will occur an intersection point $T$. Therefore, when the plane and so also $K$ is cut into two by the line $\ell$ of $\ze$ and $\ze'$, then one part -- the part of $K$ in the same halfplane as $T$ -- will be contained in the triangle $\triangle_{\ze,T,\ze'}$.

We want to underline that \emph{this part is smaller in a precise sense}, than the other, left-over part of $K$. E.g. the maximal chord in direction of $\ze'-\ze$ is $s=|\ze'-\ze| < w$ (for it cannot exceed the maximal chord of $\triangle_{\ze,T,\ze'}$ in the same direction).
Note that we are talking about the direction of $\ell$,
whence the part of $K$ in the other halfplane must have
maximal chord in this direction \emph{at least} $w$,
as the maximal chord of $K$ in any direction is at least $w$, c.f.
 \cite[Theorem 7.6.1]{Webster}.
Similarly, the part of $K$ lying in $\triangle_{\ze,T,\ze'}$
has width in the direction orthogonal to $t$
at most the height of the $\triangle$,
which does not exceed the chord $s<w$ --
while the minimal with of the totality of $K$ is $w$,
 whence the left-over part has also points at least $w$-far
from $t$, and at the same time also $\ze$ is in the boundary of this part, so the width (in this direction) of this left-over part must be at least $w$. In this sense thus it is precise if we distinguish these two sides as the ``smaller side / part of $K$" (in the same halfplane as $T$) and the ``bigger / larger side / part of $K$".

Further, considering the positive orientation of the boundary curve,
we may fix a branch of the arc length parametrization which
is continuous over the small part -- equivalently, we may apply a cut,
or fix a starting point of parametrization, in the complementary part.
In this sense the parametrization defines a unique ordering of points over the smaller part,
even if the whole boundary $\DK$ cannot be ordered.
 In the following we will always say that two points -- or their parameter values --
are in precedence according to this choice of ordering:
so that we compare only points in some unambiguously given ``smaller part" and then $\ze \prec \ze'$ has the meaning of precedence in the positively ordered arc length parametrization, used continuously along this smaller part. Also we will assume the tangent direction angle function $\alpha$ being defined according to the same continuity condition, so that $\ze \prec \ze'$ if and only if $\al(\ze)<\al(\ze')$ (or, more precisely, with $\zeta=\gamma(a)$ and $\ze'=\gamma(b)$, we have $\al(a)<\al(b)$).

As for the precedence of boundary points of $\DK$, we can equivalently say that whenever $\ze, \ze' \in \DK$ and some positively
oriented tangents to $K$ at $\ze$ and $\ze'$ are $t$ and $t'$, then we say that $\ze \prec \ze'$ if and only if the positively oriented
 halftangent of $t$ intersects the negatively oriented halftangent $t'$. Of course these tangents intersect only if they are not parallel:
but distinct parallel tangents can exists only if they are at least of distance $w$ from each other, so e.g. if the chord length $s:=|\ze-\ze'|<w$,
 then it is certainly not the case. In the case when $\ze, \ze'$ lies in a straight line segment piece of $\DK$
 (and when again either the intersection of the positively oriented halftangent of $t$ and the negatively oriented halftangent of $t'$
 is empty or conversely, the intersection of the negatively oriented halftangent of $t$ and the positively oriented halftangent of $t'$ is empty)
 then this definition of precedence also works. Finally, if $t$ and $t'$ are distinct and not parallel, then there is a unique such point $T$, and the precedence is unambiguously defined. So, defining precedence only for point pairs $(\ze,\ze') \in \DK \times\DK$ this way, it creates a partial relation in $\DK\times \DK$, which is asymmetric, but is not transitive (so we cannot consider it an ordering); yet it is quite consistent with ordering of points if we apply a certain fixed cut of the boundary and consider the ordering of points of $\DK$ accordingly.

\begin{claim}\label{cl:anglediamchord}
Let $\ze, \ze' \in \Gamma,$  $0<|\ze-\ze'|=s<w$ and let $t:=\ze+e^{i\al}\RR$ and $t':=\ze'+e^{i\al'}\RR$ be two positively oriented tangent lines at these points.

Assume that neither $t$, nor $t'$ is equal to the chord line $\ell:=\overline{\ze\ze'}$.
Then there exists a unique point of intersection $T:=t\cap t'$, moreover, we have that $T\not\in \ell$.

Furthermore, writing $H$ for the open halfplane of $\CC\setminus \ell$ with $T\in H$, and $\overline{H}$ for its closure, we also have

\medskip
(i)  $(K \cap \overline{H}) \subset \triangle:=\triangle_{\ze,T,\ze'}:=\con(\ze,T,\ze')$;

(ii) If in the triangle $\triangle~$ $~\beta:=\angle(\ze,T,\ze')= \left|\arg\left(\frac{\ze-T}{\ze'-T}\right)\right|$, then $\displaystyle \beta \ge \arcsin\left(\frac{w-s}{d}\right)$;

(iii)  $\displaystyle \diam(K\cap \overline{H})\le  \dfrac{sd}{w-s}$, and in particular $ \diam(K\cap \overline{H}) \le \dfrac{2s d}{w}$;

(iv) the arc length $\displaystyle |\Gamma \cap \overline{H}| \le  \dfrac{2sd}{w-s}$, and in particular $\displaystyle |\Gamma \cap \overline{H}| \le  \dfrac{4sd}{w}$.

\end{claim}

Note that in this fully general case $\alpha'-\alpha$ and $\sin |\alpha'-\alpha|$ can be arbitrarily small (in case $\al'$ is not much different from $\al$), but in the other direction we assert that their difference is bounded away from reaching $\pi$. In fact, even $\alpha'=\alpha$ would be possible (exactly if $[\ze,\ze']$ is a part of the boundary curve $\Gamma$ and both tangents $t,t'$ coincide with $\ell$), but for easier formulation we assume in the claim that neither $t$, nor $t'$ is $\ell$, which entails that $\al' \ne \al$. The degenerate cases when $[\ze,\ze']\subset \DK$ and some of $t, t'$ equals $\ell$ are somewhat inconvenient, for then even the assertions may fail in cases when $\ell \cap \DK$ exceeds $[\ze,\ze']$. Instead of describing these situations in an overcomplicated manner right here, we will also avoid dealing with them in the forthcoming applications of Claim \ref{cl:triangle} and Claim \ref{cl:anglediamchord} either by assuming $\overline{\ze\ze'} \cap K= [\ze,\ze']$ or by discussing concretely the cases when $t'=\ell$ or $t=\ell$.

We also note that working with the maximal chord, parallel to the chord $[\ze,\ze']$, one can get a somewhat easier way the estimate\footnote{An observation kindly offered to us by S\'andor Krenedits in personal communication.} $\beta \ge \arctan(\frac{w-s}{d})$ -- as $\arcsin$ exceeds $\arctan$, we opted for the presentation of this slightly sharper version.

\begin{proof} First, let us check that $t\ne \ell$ and $t'\ne \ell$ implies $\al\ne \al'$. For a \emph{convex}
domain and \emph{positively oriented tangents} $\al=\al'$ would be possible only if $t=t'$, while $\ze \in t$ and $\ze'\in t'$
 entails that $t=t'$ could happen only if $t,t'=\ell$, which is excluded -- so $t\ne t'$ and $\al\ne\al'$.
 Second, $t \| t'$ while $t\ne t'$ (i.e. with positive orientation, $\al'=\al+\pi \mod 2\pi$) is also impossible,
for then $K$ would have two parallel tangents with a positive distance not exceeding $s <w$,
which then would imply that $\width (K) <w$, a contradiction. So, $t$ and $t'$ are not parallel and indeed $T:=t\cap t'$ exists uniquely;
moreover, $T\not \in \ell$ is clear (for in case $T\in \ell$ either $T\ne \ze$ and so $t=\overline{T\ze}=\ell$ or $T\ne \ze'$ and
 $t'=\overline{\ze' T}=\ell$, which possibilities were both excluded by assumption). This proves the assertions about $T$ itself.

\medskip
As for (i), we have $K_0:=(K\cap H) \subset \triangle:=\con(\ze,T,\ze')$ in view of Claim \ref{cl:triangle}, so it remains to see that the same holds also for the closure $\overline{H}$ in this case. In other words, we must show additionally that $(\ell \cap K) \subset \triangle$, or, equivalently, that $(\ell \cap K) \subset [\ze,\ze']$, i.e. $(\ell \cap K) = [\ze,\ze']$. Now the tangent line $t$, not matching to $\ell$, must cut this chord line into proper halflines starting from $\ze$, with only one of which halflines containing points of $K$ -- so the said halfline must be the halfline emanating from $\ze$ towards $\ze'$. Arguing the same way for $t'$ and $\ze'$, we find that $K\cap \ell$ is covered by $[\ze,\ze']$, as stated. (Note that this latter property may easily fail if $t=\ell$ or $t'=\ell$ is allowed.)

\medskip
For the following assume, as we may, that the precedence of points $\ze, \ze'$ is chosen so that $\ze\prec \ze'$, or,
equivalently, $\al<\alpha'<\al+\pi$. Note that this is equivalent to $T$ being the intersection of the halflines
$t_{+}:=\ze+e^{i\al}\RR_{+}$ and $t'_{-}:=\ze'-e^{i\al'}\RR_{+}$.
 Therefore, in the triangle $\triangle=\triangle_{\ze,T,\ze'}$, the angle at $T$ is
$$\beta:=\angle(\ze,T,\ze'))=\arg(\ze-T)-\arg(\ze'-T) = \al+\pi-\al'=\pi-(\al'-\al)<\pi.$$
 Further, the tangent angles function can be fixed so that it changes nondecreasingly between $\al$ and $\al+\pi$, with the cut (negative jump by $-2\pi$) occurring at some point with tangent direction say $\alpha+3\pi/2~$ (mod $2\pi$).

\medskip
So, let us prove (ii). Our task is to estimate the angle $\beta$ from below: we want $\beta \ge \arcsin(\frac{w-s}{d})$. Note that $\beta$ can be close to $\pi$, even if it cannot reach it, but we claim that it cannot be too small.

For an arbitrary point $A \in \DK$ with tangent direction $\alpha(A)=\al+\pi$ (so with a tangent parallel to $t$ but oriented oppositely),
we have $\al < \arg(\ze'-\ze) < \al' < \al+\pi=\al(A)$, and $\ze\prec \ze'\prec A$.
In fact from the very definition of width it follows for the point $A$ that $a:=\dist(A,t)\ge w$,
while for boundary points $P$ with $\ze \prec P \prec \ze'$, i.e. for points of $(\Gamma \cap H) \subset (K\cap H) \subset \triangle$
we necessarily have $\dist(P,t)\le \max_{z\in\triangle} \dist(z,t)=m:=\dist(\ze',t)\le s <w,$ so that $P=A$ is not possible.

As $A \not \in \overline{H}$ (because that would entail $A\in(K\cap \overline{H})\subset \triangle$) we also find that $A \in \CC\setminus \overline{H}$, whence also $[\ze',A]\subset \CC \setminus H$.
So let us draw the chord line $f:=\overline{\ze' A}$. By convexity, for the positively oriented direction $\ff$ of the chord $f$ we have
 $\al'=\al(\ze') \le \ff=\arg(A-\ze') \le \al(A)=\al+\pi$. Note that for points $z\in f_{+}$ on the positive halfline $f_{+}:=\ze'+e^{i\ff}\RR_{+}$
 we have $\dist(z,t) \ge \dist(\ze',t)=m>0$, whence $t\cap f_{+}=\emptyset$.
On the other hand, the intersection point $C:=f\cap t$ exists uniquely, as $f$ is not parallel to $t$ (for $a:=\dist(A,t) \ne \dist(\ze',t)=m$).
So, $C\in f_{-}\cap t$, i.e. (in accordance with $\ze \prec A$) $C=f_{-} \cap t_{+}$. It follows that at $C$ the angle
$$
\theta:=\angle(\ze,C,\ze') = \arg(\ze-C)-\arg(\ze'-C)=(\al+\pi)-\ff \le \al+\pi-\al' =\beta.
$$

Consider the orthogonal projection of $\ze'$ to $t$, and denote this point by $M$: then the height of $\triangle$ at $\ze'$ is $m=|\ze'-M|$, and $0<m\le s$. Further, take also the orthogonal projection of $A$ to $t$ and denote this point by $B$: then $a=\dist(A,t)=|A-B| \ge w$.

It remains to estimate $\sin \theta$ from below. Note that the triangles $\triangle_{A,B,C}$ and $\triangle_{\ze',M,C}$
are similar triangles with a right angle at $B$ resp. $M$, whence for the angle $\angle(BCA)=\angle(MC\ze')$ at the homothety center point $C$
 we have $\sin \angle(BCA)= \frac{|A-B|}{|A-C|}=\frac{|\ze'-M|}{|\ze'-C|}$ and so also $\sin \angle(MC\ze') = \frac{a-m}{|A-\ze'|}$.
However, either $\angle(MC\ze')=\theta$ or $\angle(MC\ze')=\pi-\theta$, depending on the (both well possible) cases of $\overrightarrow{CM}$ being directed to the negative or to the positive direction of $t$, i.e. $\arg(M-C)=\al+\pi$ or $\arg(M-C)=\al$. So finally $\sin \theta= \sin(MC\ze')$ in both cases, and we are led to $\sin \theta=\frac{a-m}{|A-\ze'|}$. Therefore, using that $A,\ze' \in K$ entails $|A-\ze'|\le d$ we get that $\sin \theta \ge \frac{a-m}{d} \ge \frac{w-s}{d}$, and so in particular $\beta\ge \theta \ge \arcsin\left(\frac{w-s}{d}\right)$,  proving Part (ii).

\medskip
(iii) Using (i) we have $\diam(K\cap \overline{H}) \le \diam(K\cap \triangle_{\ze,T,\ze'}) = \max\{s, |\ze-T|, |\ze'-T|\}.$

As for $|\ze'-T|$, with the above notations and using (ii) we easily get $|\ze'-T| = m/\sin \beta \le s/ \sin \theta \le sd/(w-s)$.

At this point, however, one may apply the symmetry of the situation -- if the distance of one endpoint of the chord $[\ze,\ze']$ from $T=t\cap t'$ cannot exceed $sd/(w-s)$, then neither the distance from the other endpoint can do so: i.e. $|T-\ze|\le sd/(w-s)$ holds, too.

Consequently,  $\diam(K\cap \overline{H})\le \dfrac{sd}{w-s},$ as $s\le  \dfrac{s d}{w-s}$ is immediate.

Finally, if $0<s\le w/2$ then $\diam(K\cap \overline{H}) \le \dfrac{sd}{w-s} \le 2s \dfrac{d}{w}$ is obvious, while for $w/2<s<w$ we trivially have $\diam(K\cap \overline{H})  \le d \le 2s \dfrac{d}{w}$.

\medskip
(iv) Since $\Gamma$ is convex, the arc length of the part of $\Gamma$ in $\triangle_{\ze,T,\ze'}$ joining $\ze$ and $\ze'$ cannot exceed  the sum $|\ze-T|+|\ze'-T|$, (because it is well-known for convex curves that the included one is not longer than the including one, see e.g. \cite[page 52]{BF}). As discussed above, this can be estimated by $\dfrac{2sd}{w-s}$ and also by $4sd/w$, as claimed.
\end{proof}

In the following we will use the notation $S_z[(\al,\beta)]:=\{ z+\rho e^{i\ff} ~:~ \ff \in[(\al,\beta)]\}$ for sectors with point at $z\in\CC$ and angles between $\al,\beta$.
\begin{claim}\label{cl:plusminus}
Let $\zeta \in \partial K$ and $\bnu=-e^{i\sigma}$ be (one) outer normal vector to $K$ at $\zeta$, and $t:=\ze+e^{i\al}\RR$ be the corresponding positively oriented tangent line at $\ze$
with $\alpha=\sigma-\pi/2$. Fix any angle $0<\ff <\pi/2$.
Denote
$$
\ell_{-}:=\zeta +e^{-  \ff i}\bnu \RR
 = \ze + e^{(\sigma - \ff)i} \RR, \quad [\ze,D_{-}]:=\ell_{-} \cap K, \quad {\rm and} \quad  \de_{-}:=|D_{-}-\ze|=|\ell \cap K|.
$$
and similarly
$$
\ell_{+}:=\zeta +e^{+  \ff i}\bnu \RR
 = \ze + e^{(\sigma + \ff)i} \RR, \quad [\ze,D_{+}]:=\ell_{+} \cap K, \quad {\rm and} \quad  \de_{+}:=|D_{+}-\ze|=|\ell \cap K|.
$$

If $0<\de_{-}\le \de_{+}<w$, then any tangent line $t_{-}'$, drawn to $K$ at $D_{-}$
has negative slope with respect to $t$, i.e. $t_{-}'$ is not parallel to $t$ and the point of intersection $T=t\cap t_{-}'$ is on the halfline $\ze+e^{i(\sigma-\pi/2)}\RR_{+}$; equivalently, $\ze\prec D_{-}$ in the above discussed sense, and from the parts of $K$, arising from the cut of $\CC$ (and thus $K$) by the straight line $\ell_{-}$, the one in the sector $S_\ze[\sigma-\pi/2,\sigma-\ff]$ is the ``small part" of $K$.

Symmetrically, if $0<\de_{+}\le\de_{-}<w$, then any tangent line $t_{+}'$ drawn to $K$ at $D_{+}$ has positive slope, $D_{+} \prec \ze$, $T \in \ze-e^{i(\sigma-\pi/2)}\RR_{+}$, and from the two parts of $K$ determined by $\ell_{+}$, the small part lies in the sector  $S_\ze[\sigma+\ff,\sigma+\pi/2]$.
\end{claim}
Note that we assumed here the condition $\max(\de_{-},\de_{+})<w$; but this is not necessary. However, the slightly weaker assumption that $\min(\de_{-},\de_{+})<w/\cos\ff$, cannot be dropped: if both $\de_{\pm}\ge w/\cos\ff$, then the tangents can go in any direction (both positive or negative slope) including the possibility of being parallel to $t$. We do not discuss these because in our later application in Lemma \ref{l:fromoldproof} we will be at ease if any of the chords is as large as $w$, and so we do not need further details. Similarly, it will also be easy to deal with the case when one of $\de_{-}$ or $\de_{+}$ is $0$, whence our other assumption on $\min(\de_{-},\de_{+})>0$ is not too restrictive. Note that in case $\min(\de_{-},\de_{+})=0$, e.g. if $\de_{-}=0$, then also $\ell_{-}$ is tangent to $K$ (as it does not contain any interior points, only $\ze\in\DK$), thus $K$ lies entirely in some of the sectors lying above $t$ and determined by the line $\ell_{-}$; however, we cannot always tell which side is the small resp. large side, as any of these two sectors $S_\ze[\sigma-\pi/2,\sigma-\ff]$ or $S_\ze[\sigma-\ff,\sigma+\pi/2,]$ may contain $K$. Of course, if the other chord is nonzero, i.e. $\de_{+}>0$, then clearly that side, i.e. the latter sector, will contain $K$. The situation is similar if we start with $\de_{+}=0$.

\begin{proof} By symmetry, we may, hence we will assume $0 < \de_{-}\le\de_{+}<w$.

It is clear that the tangent $t_{-}'$ cannot be parallel to $t$, for in this case we would have $K$ contained between the \emph{distinct} parallel supporting lines $t$ and $t_{-}'$
of distance $(0<)\de_{-}\cos(\ff) <\de_{-}<w$, a contradiction. Now if $t_{-}'$ was to have a positive slope, i.e. $T=t\cap t_{-}'$ falling on the halfline $\ze-e^{i(\sigma-\pi/2)}\RR_{+}$,
then obviously we had $D_{+}$ below this tangent, and $\de_{+} < \de_{-}$, contrary to assumption.

So it remains the only possibility $t_{-}'$ having negative slope. That is, $T\in \ze+e^{i(\sigma-\pi/2)}\RR_{+}$, $\ze\prec D_{-}$, and the above Claim~\ref{cl:triangle} applies. It means that the triangle $\triangle_{\ze,T,D_{-}}$ covers the part of $K$ in the respective sector $S_\ze[\sigma-\pi/2,\sigma-\ff]$, whence this can only be the ``small part" of $K$.
\end{proof}

\section{Technical preparations for the investigation of $L^q(\partial K)$ norms}\label{s:Lemmas}

\begin{lemma}\label{l:Nikolskii} For any polynomial of degree at most $n$ we have that
\begin{equation}\label{eq:Nikolskii}
\|p\|_{L^q(\DK)} \ge \left( \frac{d}{2(q+1)}\right)^{1/q}
~\|p\|_{L^\infty(\DK)} ~ n^{-2/q} .
\end{equation}
\end{lemma}

For a proof of this Nikolskii-type estimate, see \cite[Lemma 1]{PR}.

Next, let us define the subset $\HH:=\HH_K^q(p) \subset \DK$ the
following way.
\begin{equation}\label{eq:Hsetdef}
\HH:=\HH_K^q(p):=\{\zeta\in \partial K~: ~ |p(\zeta)| > c n^{-2/q} \|p\|_\infty\}
\quad \left(c:=\frac12 \left({8\pi(q+1)} \right)^{-1/q} \right).
\end{equation}
Then in \cite[Section 3.1]{PR} it was deduced from the above Lemma that
we have
\begin{lemma}\label{l:Hlogp} Let $\HH \subset \DK$ be defined according to \eqref{eq:Hsetdef}. Then for all $p \in \PP_n$ we have
\begin{equation}\label{eq:pqintegralonH}
\int_{\HH} |p|^q \geq \frac12 \|p\|^q_{L^q(\DK)}.
\end{equation}
Furthermore, for any point $\ze \in \HH$, and for any $p\in \PK$ we
also have
\begin{equation}\label{eq:pnormperponH}
\log \frac{\|p\|_\infty}{|p(\ze)|} \le \log(16\pi) + 2 \log n ~~
(\forall n \in \NN),\quad
\log \frac{\|p\|_\infty}{|p(\ze)|} \le
 \frac{107}{40} \log n ~~  ( n\ge 73).
\end{equation}
\end{lemma}

The other key and innovative feature of the original work of Er\H od
was invoking Chebyshev's Lemma, which we recall here.

\begin{olemma}{\bf (Chebyshev).}\label{l:capacity}
Let $J=[u,v]$ be any interval on the complex plane with $u\ne
v$. Then for all $k\in\NN$ we have
\begin{equation}\label{capacity}
\min_{w_1,\dots,w_k\in \CC} \max_{z\in J} \left| \prod_{j=1}^k
(z-w_j) \right| \ge 2 \left(\frac{|J|}{4}\right)^k~.
\end{equation}
\end{olemma}

Actually, we will also use this lemma in the next slightly more
general form of an estimation using the transfinite diameter.
\begin{olemma}{\bf (Transfinite Diameter Lemma).}\label{l:FeketeSzego}
Let $K\Subset \CC$ be any compact set and $p\in\PK$ be a monic polynomial,
i.e. assume that $p(z)=\prod_{j=1}^n (z-z_j)$ with all $z_j\in K$.
Then we have $\|p\|_{\infty} \ge \Delta(K)^n$.
\end{olemma}

\begin{proof} Lemma \ref{l:capacity} is essentially the classical result of
Chebyshev for a real interval \cite{Chebyshev}, cf. \cite[Part 6, problem 66]{PS}, \cite{BE, MMR}.
The form with the transfinite diameter was first proved in various forms by Fekete, Faber and Szeg\H o. For details and references see \cite[Lemma P]{PR} and its discussion there.
\end{proof}

In the below proofs we will need the following straightforward
calculation of the type usually considered in connection with
transfinite diameter.

\begin{lemma}\label{l:transfdiam} Let $K'\Subset K\Subset \CC$ be two compact sets
 with diameters $d':=\diam K'$ and $d:=\diam K$, and assume $d'\le d/k$
with some parameter $k>10$, say. Then if $p \in \PK$ has $m \ge \dfrac{3\log 2}{\log k} n$
zeros in $K'$, then $\|p\|_{K'} < 2^{-n} \|p\|_{K}$.
\end{lemma}
\begin{proof} Assume, as we may, that the leading coefficient of $p$ is just 1, and so $p(z)=\prod_{j=1}^n (z-z_j)$.
It is well-known, see e.g. \cite{PR2} or \cite[\S
1.7.1.]{RansSur}\footnote{However, note a disturbing misprint in this
fundamental reference: in \S 1.7.2. the first two displayed formulas
must be corrected to have the opposite direction of the inequality
sign.}, that the capacity, or transfinite diameter of a compact set is
at least its diameter divided by 4 (and is, on the other hand, at most
the diameter divided by 2). Using this or directly Chebyshev's Lemma,
we certainly have $\|p\|_K (\ge \Delta(K)^n ) \ge (d/4)^n$.

Estimating from the other side, we have for any point $z'\in K'$ the
estimate $|p(z')| \le d'^m d^{n-m}$, whence $\|p\|_{K'} \le d'^m
d^{n-m}$ and after dividing these two estimates we get
$$
\frac{\|p\|_{K'}}{\|p\|_{K}} \le \frac{d'^m d^{n-m}}{(d/4)^n} = 4^n \left( \frac{d'}{d}\right)^m \le 4^n k^{-m} = 2^{2n- m \frac{\log k}{\log 2}} \le 2^{2n- 3 n} =2^{-n}.
$$
\end{proof}

\section{  A refined estimate by tilting the normal line}\label{sec:tiltednormal}

The method in our recent works \cite{PR, PR2} was to consider an upper
subinterval $J\subset [\ze,D]:=\bnu \cap K$, with $\bnu$ a normal line at
$\ze\in\DK$, apply a suitable classification of zeros and for the say
$k$ zeroes lying close to $J$, select a maximum point $\tau_0$ of the
corresponding product of the respective $k$ terms $(z-z_j)$. This
direct approach can be used to get some general infinity norm estimates
(in fact: an order $n^{2/3}$ lower estimation \cite{Rev1}) even if the
depth may tend to zero. Also, we succeeded to obtain the right order
(i.e., order $n$) lower estimate for some \emph{special} classes of
domains in \cite{PR, PR2}. However, this method incorporates some
losses with respect to depth, and for fully general cases there seems
to be no way to obtain optimal, or close-to-optimal order by this method.

Instead, here we pursue an essentially modified method, based on an
insightful idea of G. Hal\'asz and exploited, for the case of the
maximum norm, in the proof of Theorem \ref{th:convexdomain} in
\cite{Rev2}. For more explanations and the heuristical reasons for the
key idea of tilting the normal line in this approach, the interested
reader may consult \cite{Rev2, SofiaCAA}.

In the main proof in \cite{Rev2} one could make use of the maximality
of $|p(z)|$: as before in \cite{PR,PR2}, we now have to take a general
boundary point and derive pointwise estimates in this more general
case.

Here we work out the following version of the main proof from
\cite{Rev2}.

\bigskip
\begin{lemma} [Tilted normal estimate]\label{l:fromoldproof} Let $\zeta \in \partial
K$ and $\bnu=-e^{i\sigma}$ be (one) outer normal vector to $K$ at
$\zeta$. Fix the angles
\begin{equation}\label{fidef}
\psi:=\arctan \left( w/d \right)\in (0,\pi/4] \qquad
\text{\rm and} \qquad \theta:=\psi/20 \in (0,\pi/80].
\end{equation}
Denote
$$
\ell_{\pm}:=\zeta +e^{\pm 2\theta i}\bnu \RR
 = \ze + e^{(\sigma \pm 2\theta)i} \RR, \quad [\ze,D_{\pm}]:=\ell_{\pm} \cap K, \quad {\rm and} \quad  \de_{\pm}:=|D_{\pm}-\ze|=|\ell \cap K|,
$$
with the two alternatives with respect to $\pm$ understood separately. Then we have the following.
\begin{enumerate}[(i)]
\item If $\ell_{+} \cap \intt K = \emptyset$ or $\ell_{-} \cap \intt K = \emptyset$ -- in particular~ if either $\de_{-}=0$, i.e. $D_{-}=\ze$ and $\ell_{-} \cap K = \{\ze\}$,  or $\de_{+}=0$, i.e. $D_{+}=\ze$ and $\ell_{+} \cap K = \{\ze\}$ --  then
    $$
    \left| \dfrac{p'}{p}(\ze) \right| \ge \dfrac{1}{2d} n.
    $$
\item If both $\ell_{\pm} \cap \intt K \ne \emptyset$ -- entailing that $\de_{\pm}>0$  --  and $0<\delta_\pm<w$ then it holds
\begin{equation}\label{eq:oldyield}
\left| \frac{p'}{p}(\ze)\right| > 0.001 \frac{w}{d^2} n -
\frac{2}{39 \de_{\pm}} \log \frac{\max_{K\cap\ell_{\pm}} |p|}{|p(\ze)|} \ge 0.001 \frac{w}{d^2} n - \frac{2}{39 \de_{\pm}} \log \frac{\|p\|_\infty}{|p(\ze)|},
\end{equation}
where the choice of sign has to be such that $\de_{\pm}=\min(\de_{-},\de_{+})$.
In particular, if $\ze\in \HH$ -- defined in \eqref{eq:Hsetdef} --
and $n\ge 73,$ then according to the last estimate of \eqref{eq:pnormperponH}
\begin{equation}\label{eq:oldyield1}
\left| \frac{p'}{p}(\ze)\right| > 0.001 \frac{w}{d^2} n -
\frac{0.15}{\de_{\pm}}\ln n.
\end{equation}

\item Finally, if $\max(\de_{-},\de_{+}) \ge w/2$,
then the above estimates \eqref{eq:oldyield}, \eqref{eq:oldyield1} hold
even for both choices of sign,
so also with the one providing $\max(\de_{-},\de_{+})$,
irrespective of the size of the various parts of $K$ as
cut by the chord lines or if $\intt  K\cap \ell_{\pm} =\emptyset$
or not.
\end{enumerate}
\end{lemma}
\begin{proof} Assume, as we may, $\ze=0$ and $\bnu=\bnu(\zeta)=\bnu(0)=-i$, i.e. the selected supporting line is the real line $\RR$ (oriented positively) and $\sigma=\pi/2$. So, $K$ lies in the upper halfplane: $K \subset \{z~:~\Im z \ge 0\}$.

Consider now the situation in (i) -- e.g. let us consider the case when $\intt K \cap \ell_{-} =\emptyset$, the other case being symmetrical.
The ray (straight half-line) $e^{i(\pi/2- 2\theta)} \RR_{+}=\ell \cap \{z~:~\Im z\ge 0\}$,
 emanating from $\zeta=0$ in the direction of $e^{i(\pi/2 - 2\theta)}$ intersects $K$ in the segment $[0,D]$,
 and if $\ell_{-} \cap \intt K= \emptyset$, then we necessarily have $[0,D]\subset \DK$.
So, $\ell_{-}$ is a supporting line of $K$, and either $K \subset S[0,\pi/2 - 2\theta]$ or $K \subset S[\pi/2 -2\theta,\pi]$.
In either case a standard argument using e.g. Tur\'an's Lemma \ref{Tlemma} yields directly $|p'(\ze)/|p(\ze)|\ge n/(2d)$.
Hence part (i) is proved.

It remains to discuss the cases when $\intt K \cap \ell_{\pm} \ne \emptyset$, entailing that both $\de_{\pm}>0$.

Again we choose to deal with one of the two entirely symmetrical cases and suppose   $0<\de_{-}\le\de_{+}<w$, if  $\min(\delta_-, \delta_+)<w/2$, and $0<\de_{+}\le\de_{-}$ otherwise. Therefore,  we can take $\delta_-$ in both cases (ii) and (iii). To further ease notation, we will drop the minus sign from the index and will simply write $\de$, $D$,  etc. for the previously given $\de_{-}$, $D_{-}$ in the rest of the argument.

The small geometrical claim the proof of which ramifies here is the statement that we necessarily have
\begin{equation}\label{sgc}
|z|\le 2\de d/w  \quad \textrm{for} \quad z \in (K\cap S[0,\theta]).
\end{equation}
This is clearly true if $\de \ge w/2$, because $|z|=|z-\ze|\le d$.
However, if $\de<w/2$, then Claim \ref{cl:plusminus}
 applies with $\ff:=2\theta$, which in turn furnishes $\diam(K\cap S[0,\pi/2-2\theta]) \le 2\de d/w$, according to
Claim \ref{cl:anglediamchord} (iii). As $S[0,\theta] \subset S[0,\pi/2-2\theta]$,
it is all the more true that $\diam(K\cap S[0,\theta]) \le 2\de d/w$; so again $|z|=|z-\ze|\le \diam(K\cap S[0,\theta]) \le 2\de  d/w$,
as wanted. This small statement will be soon used in the calculations with points of the forthcoming set $\Z_1$.

Denote by $\Z:=\{z_j=r_je^{i\ff_j}~:~j=1,\dots,n\}$ the $n$-element set of zeroes (listed according to multiplicities) of the fixed polynomial $p\in \PK$. Note that $0\le \ff_j\le \pi$ ($j=1,\dots,n$). 

Observe that for any subset $\W\subset\Z$ we have for $M:=\left|\dfrac{p'}{p}(\ze)\right|$ that
\begin{equation}\label{Impartsum}
M = \left| \frac{p'}{p}(0) \right| \ge -\Im {\frac{p'}{p}(0)} =
\sum_{j=1}^n \Im \frac {-1}{z_j} \ge \sum_{z_j\in\W} \Im \frac
{-1}{z_j}=\sum_{z_j\in\W} \frac {\sin \varphi_j}{r_j} \,,
\end{equation}
because all terms in the full sum are nonnegative.

The segment $J$ is defined to be
\begin{equation}\label{Jdef}
J:= \left[\frac{\ze+3D}{4}, D \right] =
J:=\{\tau:=te^{i(\pi/2-2\theta)}\de~:~3/4 \le t \le 1\}.
\end{equation}
Clearly, by convexity we have $J\subset K$.

Denoting $B_r(0):=\{z\,:\, |z|\le r\}$ and writing $\Z[(\alpha,\beta)]:=\Z \cap S[(\alpha,\beta)]$, we split the set $\Z$ into the
following parts.
\begin{align}\label{Zsplitup}
\Z_1:&= \Z[0,\theta]\,, \qquad \qquad \qquad \qquad
\qquad\qquad\qquad \qquad \qquad \qquad &\mu:=\#\Z_1, \notag \\
\Z_{2}:&= \Z(\theta,\pi-\theta)\cap \left\{\Im (e^{i2\theta}z)<
\frac 38 \delta \right\}\,,&\nu:=\#\Z_{2}, \notag \\
\Z_{3}:&= \Z(\theta,\pi-\theta)\cap\left\{\Im (e^{i2\theta}z) \ge
\frac 38 \delta \right\} \cap B_{\frac54\delta}(0)\,,&\kappa:=\#\Z_{3}, \notag \\ 
\Z_{4}:&= \Z(\theta,\pi-\theta)\cap\left\{\Im (e^{i2\theta}z) \ge
\frac 38 \delta \right\}\setminus B_{\frac54\delta}(0)= \\ 
    &=\Z(\theta,\pi- \theta)\setminus \left( \Z_{2}\cup \Z_{3}\right)\,,
    &k:=\#\Z_{4}, \notag \\ 
\Z_5:&= \Z[\pi-\theta,\pi]\,,\, &m:=\#\Z_5. \notag
  \end{align}

\begin{figure}
\begin{center}
 \includegraphics[trim={5mm 193mm 40mm 0mm},clip,width=10cm,keepaspectratio]{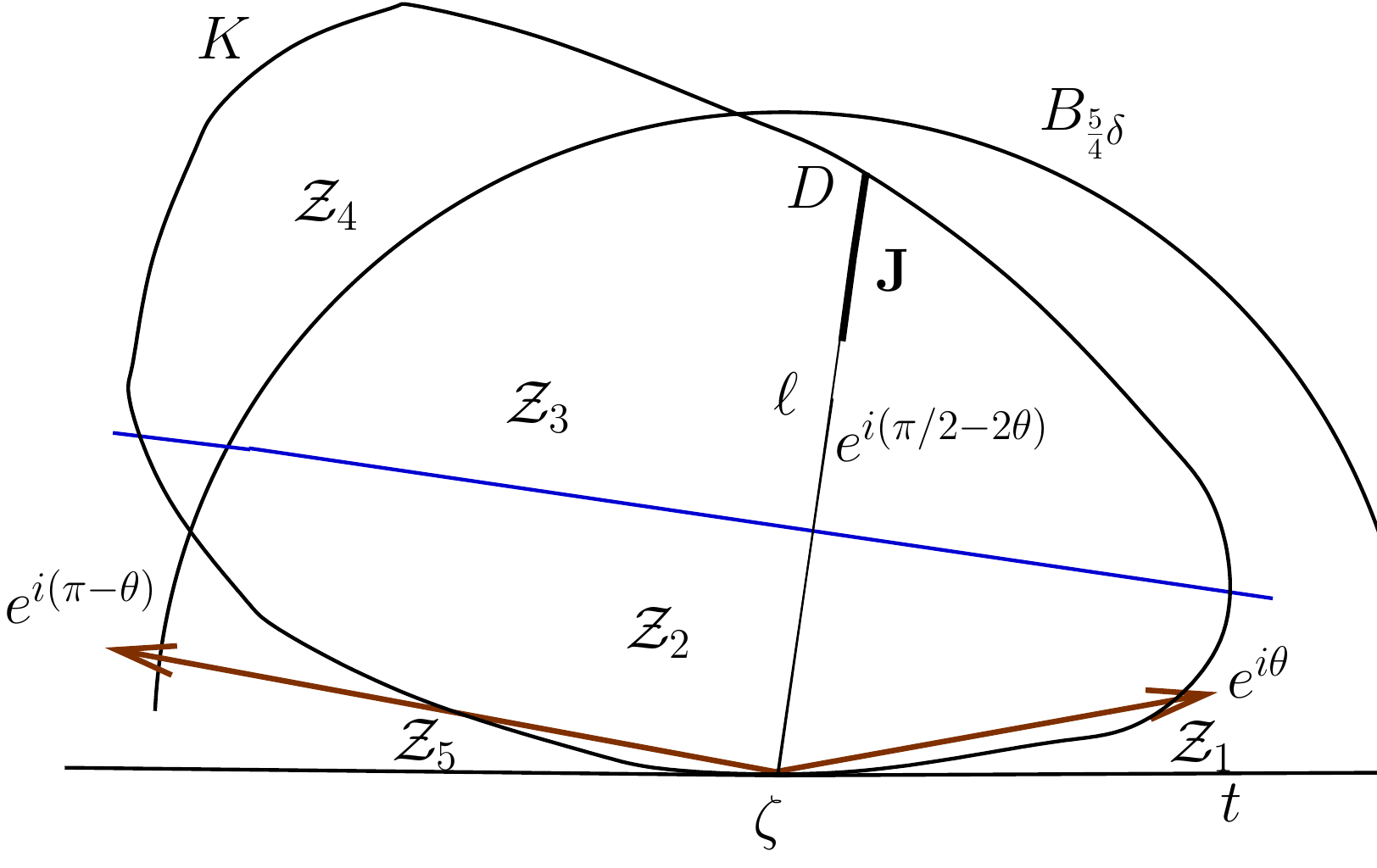}\\
  \caption{The classification of zeros according to location}\label{Fig-Zero}
\end{center}
\end{figure}

In the following we estimate $\left|\dfrac{p(\tau)}{p(\zeta)}\right|$
from below.

First we estimate the distance of any $z_j\in\Z_1$ from $J$.
In view of the above discussed small claim~\eqref{sgc}, for any
$z=re^{i\varphi}\in (K\cap S[0,\theta])$
 we have $|z| \le \dfrac{2\de d}{w}$,
whence from convexity of the tangent function
\begin{equation}\label{deltarst}
r\sin\theta \le \frac{2\de d}{w }
\sin \theta \le  2 \delta \frac {d}{w} \tan \theta = 2 \delta \frac{\tan\theta}{\tan(20\theta)} < \frac{\delta}{10} ~.
\end{equation}

Now $\dist (z,J)=\min_{3/4\le t\le 1} |z-\tau|$ (where
$\tau:=te^{i(\pi/2-2\theta)}\de$), and by the cosine theorem
$|z-\tau|^2=r^2+t^2\delta^2-2rt\delta\cos(\pi/2-\varphi-2\theta)$.
Because of
$\cos(\pi/2-\varphi-2\theta)=\sin(\varphi+2\theta)\le\sin(3\theta)\le 3
\sin\theta$, \eqref{deltarst} implies
$$|z-\tau|^2  =r^2+ 10 t^2\delta  r\sin \theta -6 t\delta r \sin\theta =
r^2+\left(10 t^2- 6t\right)\delta r\sin \theta.
$$ and
thus $\displaystyle \min_{3/4\le t\le 1} |z-\tau|^2 =|z-\tau|^2  \big|_{t=3/4}
= r^2+ \frac9{8} \delta r\sin \theta $. It
follows that we have
$$
\frac{|z-\tau|^2}{|z|^2}\ge 1+ \frac98 \frac{\sin\theta\,\,\delta}{r} > 1+
\frac98 \frac{\sin\theta\,\,\delta}{d} \quad \left(\tau\in J\right)\,.
$$

Now $\de/d\le 1$ and $\sin\theta < \pi/80 <0.1$, hence we can apply
$\log(1+x)\ge x-x^2/2 \ge 0.9 x$ for $0<x<0.1$ to get
$$
\frac{|z-\tau|^2}{|z|^2} \ge \exp\left(0.9 \frac{9\sin\theta ~
\delta}{8d}\right) > \exp\left(\frac{\sin\theta ~ \delta}{d}
\right) \quad \left(\tau \in J\right)~.
$$
Applying this estimate for all the $\mu$ zeroes $z_j\in\Z_1$ we finally
find
\begin{equation}\label{Z1zeros}
\prod_{z_j\in\Z_1}\left|\frac{z_j-\tau}{z_j}\right| \ge
\exp\left(\frac12 \frac{\sin\theta ~ \delta\mu}{d} \right)~\qquad
\left(\tau=t\de e^{i(\pi/2-2\theta)} \in J \right).
\end{equation}

The estimate of the contribution of zeroes from $\Z_5$ is somewhat
easier, as now the angle between $z_j$ and $\tau$ exceeds $\pi/2$. By
the cosine theorem again, we obtain for any $z=r e^{i\varphi}\in
S[\pi-\theta,\pi]\cap K$ the estimate
\begin{align}\label{franc}
|z-\tau|^2 = & r^2+t^2\de^2-2\cos(\varphi-(\pi/2-2\theta))\,rt\delta
\notag\\ \ge & r^2+t^2\de^2+2\sin\theta ~ rt\delta
> r^2\left(1+\frac{3\sin\theta ~ \delta}{2d}\right) ~~~\left(\tau \in J\right)~,
\end{align}
as $t\ge 3/4$ and $r\le d$. Hence using again $\de/d\le 1$ and
$1.5\sin\theta < 1.5 \pi /80 < 0.1$ we can again apply $\log(1+x)\ge 0.9 x$ for $0<x<0.1$ to get
$$
\frac{|z-\tau|}{|z|} \ge \exp\left(\frac12 0.9 \frac{3\sin\theta ~ \delta}{2d}\right) >
\exp\left(\frac{\sin\theta\,\delta}{2d} \right) ~~~\left(\tau \in J\right)~,
$$
which then yields
\begin{equation}\label{Z4contri}
\prod_{z_j\in\Z_5}\left|\frac{z_j-\tau}{z_j}\right| \ge\exp\left(
\frac{\sin\theta\,\,\delta m}{2d} \right) ~\qquad
\left(\tau=t\de e^{i(\pi/2-2\theta)} \in J \right).
\end{equation}

Observe that zeroes belonging to $\Z_{2}$ have the property that they
fall to the opposite side of the line $\Im(e^{i2\theta}z)=3\de/8$ than
$J$, hence they are closer to $0$ than to any point of $J$. It follows
that
\begin{equation}\label{Zstarcontri}
\prod_{z_j\in\Z_{2}}\left|\frac{z_j-\tau}{z_j}\right| \ge 1
~\qquad \left(\tau=t\de e^{i(\pi/2-2\theta)} \in J \right).
\end{equation}
Next we use Chebyshev's Lemma \ref{l:capacity} to estimate the contribution of zero
factors belonging to $\Z_{3}$. We find
\begin{equation}\label{Z2pcontri}
\max_{\tau\in J} \prod_{z_j\in\Z_{3}}
\left|\frac{z_j-\tau}{z_j}\right| \ge
2\left(\frac{|J|}{4}\right)^{\kappa^{}} \prod_{z_j\in\Z_{3}}
\frac{1}{r_j} \ge \left(\frac{1}{20}\right)^{\kappa^{}} >
\exp(-3 \kappa^{})~,
\end{equation}
in view of $|J|=\de/4$ and $r_j \le \frac54 \de$ and using $\log 20 = 2.9957\dots <3$.

Note that for any point $z=re^{i\varphi}\in B_{\frac54\delta}(0)\cap
\{\Im(e^{i2\theta}z)\ge 3\de/8 \}$ we must have
$$
\frac{3\delta}{8}\le
\Im(e^{i2\theta}re^{i\varphi})=r\sin(\varphi+2\theta)~,
$$
hence by $r\le \frac54 \de$ also
$$
\sin(\varphi+2\theta)\ge \frac{3\delta}{8r} \ge \frac 3{10}
$$
and $\sin \varphi\ge \sin(\varphi+2\theta)-2\theta \ge 3/10 - \pi/40 >
1/5$. Applying this for all the zeroes $z_j\in \Z_{3}$ we are led to
\begin{equation}\label{maxcontri}
1\le \frac{\frac54 \de}{r_j} \le \frac{25}{4} \de \sj \qquad
\left(z_j\in\Z_{3}\right)~.
\end{equation}
On combining \eqref{Z2pcontri} with \eqref{maxcontri} and writing in
$3\cdot \dfrac{25}{4} < 19$ we are led to
\begin{equation}\label{Z2plus}
\max_{\tau\in J} \prod_{z_j\in\Z_{3}} \left|\frac{z_j-\tau}{z_j}\right|  > \exp\left(- 19 \de
\sum_{z_j\in\Z_{3}}\sj \right)~.
\end{equation}

Finally we consider the contribution of the zeroes from $\Z_{4}$, i.e.
the ``far" zeroes for which we have $\Im(z_je^{2i\theta})\ge3\de/8$,
$\varphi_j\in (\theta,\pi-\theta)$ and $|r_j|\ge \frac54 \de$. Put now
$Z:=z_je^{2i\theta}=u+iv=re^{i(\varphi_j+2\theta)}$, and $s:=|\tau|=t\de$, say. We
then have
\begin{align}\label{farupzeroes}\notag
\left|\frac{z_j-\tau}{z_j}\right|^2 &= \frac{|Z-t\de i|^2}{r^2}=
\frac{u^2+(v-s)^2}{r^2}=1-\frac{2vs}{r^2}+\frac{s^2}{r^2} \\ &>
1-\frac{2vs}{r^2}+\frac{s^2}{r^2}\frac{v^2}{r^2} =
\left(1-\frac{vs}{r^2}\right)^2\ge
\left(1-\frac{|v|\delta}{r^2}\right)^2=
\left(1-\frac{\de|\sin(\varphi_j+2\theta)|}{r}\right)^2.
\end{align}
Recall that $\log(1-x) >-x-\frac{x^2}{2}\frac {1}{1-x}\ge -3~x$
whenever $0\le x \le 4/5$. We can apply this for
$x:=\de|\sin(\varphi_j+2\theta)|/r_j\le \de/r_j \le 4/5$ using $r=r_j=|z_j|=|u+iv|
\ge \frac54 \de$. As a result, \eqref{farupzeroes} leads to
\begin{equation}\label{farupzeroescontr}
\left|\frac{z_j-\tau}{z_j}\right|\ge \exp\left( - 3
\de\frac{|\sin(\varphi_j+2\theta)|}{r_j}\right)~,
\end{equation}
and using $|\sin(\varphi_j+2\theta)|\le \sin(\varphi_j)+\sin(2\theta)
\le 3 \sin\varphi_j$ (in view of $\varphi_j\in(\theta,\pi-\theta)$),
finally we get
\begin{equation}\label{Z3plusfin}
\prod_{z_j\in\Z_{4}}\left|\frac{z_j-\tau}{z_j}\right| \ge
 \exp\left( -9\de \sum_{z_j\in\Z_{4}}
 \frac{\sin\varphi_j}{r_j}\right)\qquad \left(\tau=t\de
e^{i(\pi/2-2\theta)} \in J \right)~.
\end{equation}

If we collect the estimates  \eqref{Z1zeros} \eqref{Z4contri}
\eqref{Zstarcontri} \eqref{Z2plus}  and \eqref{Z3plusfin}, we find for
a certain point of maxima $\tau_0\in J$ in \eqref{Z2plus} the
inequality
$$
 \frac{|p(\tau_0)|}{|p(0)|} = \prod_{z_j\in\Z
}\left|\frac{z_j-\tau_0}{z_j}\right| >
\exp\left\{\frac{1}{2} \sin\theta\,\delta\frac{\mu+m}{d}- 19 \de
\sum_{z_j\in \Z_{2}\cup\Z_{3}\cup\Z_{4}} \sj \right\}~,\notag
$$
or, after taking logarithms and canceling by $\de/2$
\begin{equation}\label{pointofmaxima}
\frac{2}{\de} \log\left| \frac{p(\tau_0)}{p(0)}\right| \ge (\mu+m) \frac{\sin\theta}{d} - 38 \sum_{z_j\in
\Z_{2}\cup\Z_{3}\cup\Z_{4}} \sj .
\end{equation}
Observe that for the zeroes in $\Z_2\cup\Z_3\cup\Z_4$ we have
$\sin\varphi_j > \sin \theta$, whence also
\begin{equation}\label{triviesti}
(\nu+\kappa+k)\frac{\sin\theta}{d} - \sum _{z_j\in
\Z_2\cup\Z_3\cup\Z_4} \sj \le 0~.
\end{equation}
Adding \eqref{triviesti} to the right hand side of \eqref{pointofmaxima} and taking into
account $\#\Z=\sum_{j=1}^5 \#\Z_j$, we obtain
\begin{equation}\label{almostfinal}
\frac{2}{\de} \log\left| \frac{p(\tau_0)}{p(0)}\right| \ge  \frac{\sin\theta}{d} ~n
- 39 \sum _{z_j\in \Z_2\cup\Z_3\cup\Z_4} \sj~.
\end{equation}
Making use of \eqref{Impartsum} with the choice of $\W:=
\Z_2\cup\Z_3\cup\Z_4$ we arrive at
$$ 
\frac{2}{\de} \log\left| \frac{p(\tau_0)}{p(0)}\right| \ge \frac{\sin\theta}{d} ~n - 39 \left| \frac{p'}{p}(0)\right| ~,
$$
that is, writing in again the normalization $\zeta:=0$,
\begin{equation}\label{Mnesti}
M:=\left| \frac{p'}{p}(\ze)\right| > \frac{1}{39}\frac{\sin\theta}{d} n - \frac{2}{39\de} \log\left| \frac{p(\tau_0)}{p(\ze)}\right| .
\end{equation}
It remains to recall \eqref{fidef} and to estimate
$$
\sin\theta = \sin \left( \frac{\arctan(w/d)}{20}\right)~.
$$
As $\theta\in (0,\pi/80]$, \; $\sin\theta > \theta(1-\theta^2/6)\ge
\theta(1-\pi^2/38400)>0.999\,\theta$ and as $0<w/d\le 1$, \;
$\arctan(w/d)\ge (w/d) (\pi/4)$, whence
$$
\sin\theta \ge 0.999 \frac{\arctan(w/d)}{20} \ge \frac{0.999 \pi}{80} \frac wd~ > 0.039 \frac wd~.
$$
If we substitute this last estimate into \eqref{Mnesti} we get
$$
M:=\left| \frac{p'}{p}(\ze)\right| > 0.001 \frac{w}{d^2} n - \frac{2}{39\de} \log\left| \frac{p(\tau_0)}{p(\ze)}\right|,
$$
and the Lemma obtains.
\end{proof}

\section{A combined estimate for values of the logarithmic derivative}\label{sec:CornishVari}

\begin{lemma}\label{l:twosides} Let $\zeta, \zeta' \in \partial K$, $\ze\prec \ze'$ and let (some) tangents (supporting lines to $K$)
 be given at these points as $t:=\zeta+e^{i\alpha}\RR$ and $t':=\zeta'+e^{i\alpha'}\RR$, respectively,
with the directional vectors $e^{i\alpha}$ and $e^{i\alpha'}$ oriented positively and $\alpha < \alpha'<\alpha+\pi$. Define the angle $\beta:=\pi-(\alpha'-\alpha)$ and write $s:=|\ze'-\ze|$. Then if $s\le s_0:=s_0(\beta):=\dfrac{\min(1,2\sin\beta)}{384} d$, then for any $p\in \PK$ we have the following alternative.
\begin{enumerate}[(i)]
\item Both $|p(\ze)|, |p(\ze')| \le 2^{-n} \|p\|_\infty$; in particular for $n\ge 15$ both $\ze,\ze' \not\in \HH$;
\item or $\displaystyle\left| \frac{p'}{p}(\ze) \right| + \left|
    \frac{p'}{p}(\ze') \right| \ge \frac{3\sin\beta}{8d} n$.
\end{enumerate}
\end{lemma}
\begin{proof}
Denote $T:=t\cap t'$. Then by the above notations, we have $\angle(\ze
T\ze')= \beta$. Moreover, since $t, t'$ are tangents of $K$, we have
$K\subset S$, where $S=S_T[\alpha',\pi-\alpha]$ is the sector
with point $T$ and containing the chord $[\ze,\ze']$.

Let now $z\in K \subset S$ be arbitrary. We can describe the location
of $z$ with respect to each one of the three points
$\ze,\ze',T$. Let us write $r:=|z-\ze|$, $r':=|z-\ze'|$ and
$\rho:=|z-T|$, and let the angles be $\angle(z\ze T)=\ff$, $\angle(z\ze' T)=\ff'$, and
$\angle(zT\ze)=\phi$, $\angle(zT\ze')=\phi'$. Then of course $\phi
+\phi'=\beta$. Now if the distances (heights) of $z$ from the tangents
are $m:=\dist(z,t)$ and $m':=\dist(z,t')$, then we have $m=r\sin \ff=
\rho\sin\phi$ and $m'=r\sin\ff'=\rho\sin\phi'$. If we further take
$\rho \ge R:=3\max (|T-\ze|,|T-\ze'|)$, then we also have $r,r'\le d$ and $r
\le \rho+|\ze-T| \le (4/3)R \rho$, $r' \le \rho+|\ze'-T| \le (4/3) \rho$ whence
\begin{align*}
\frac{\sin\ff}{r}+\frac{\sin\ff'}{r'}&= \frac{r\sin\ff}{r^2}+\frac{r'\sin\ff'}{r'^2}
= \frac{ \rho\sin\phi}{r^2}+\frac{\rho\sin\phi'}{r'^2}
\\& \ge \frac{3}{4d} \left(\sin\phi+\sin\phi' \right)
= \frac{3}{2d} \sin\left(\frac{\phi+\phi'}{2} \right) \cos\left(\frac{\phi-\phi'}{2} \right)
\\&\ge \frac{3}{2d} \sin\left(\frac{\beta}{2} \right) \cos\left(\frac{\beta}{2} \right)
= \frac{3\sin\beta}{4d} .
\end{align*}
Denote now the subset of zeroes of $K$ which lie at least $R$ far from
$T$ as $\Z(R)$, i.e. write $\Z(R):=\Z\setminus B_R(T)$. Then for any
$z_j\in \Z(R)$ we have $|z_j-T|\ge R$ and $z_j\in K \subset S$, i.e. the above calculation is valid, and we obtain that
$$
\frac{\sin(\arg(z_j-\ze)-\alpha)}{|z_j-\ze|}+\frac{\sin(\arg(z_j-\ze')-\alpha')}{|z_j-\ze'|}
\ge \frac{3\sin\beta}{4d}.
$$

Observe that the terms here are the general terms for the expression
$\Im \left( \frac{e^{-i\alpha}}{\ze-z_j}\right)  + \Im
\left(\frac{e^{-i\alpha'}}{\ze'-z_j}\right)$, whence denoting $\nu:=\#
\Z(R)$ we get similarly to \eqref{Impartsum}
\begin{align}\label{twosidestogether} \notag
\left| \frac{p'}{p}(\ze) \right| + \left| \frac{p'}{p}(\ze') \right|
& \ge \Im \left( \frac{p'}{p}(\ze) e^{-i\alpha} \right)
+ \Im \left( \frac{p'}{p}(\ze') e^{-i\alpha'} \right)
\\& = \sum_{j=1}^n \left\{ \Im \left( \frac{e^{-i\alpha}}{\ze-z_j}\right)
+ \Im \left(\frac{e^{-i\alpha'}}{\ze'-z_j}\right) \right\}
\ge \sum_{z\in \Z(R)} \frac{3\sin\beta}{4d} = \frac{3\sin\beta}{4d} \nu.
\end{align}

Next we would like to estimate the number $\mu:=n-\nu$ of zeroes of a
fixed $p\in\PK$ in $\Z\setminus \Z(R)$, i.e. in $\Z^{*}:=\Z\cap
B_R(T)$. Recall that $K\subset S$, whence also $\Z^{*} \subset
K^{*}:=(K\cap B_R(T) ) \subset (S \cap B_R(T)) =:S^{*}$. For the
sectorial part $S^{*}$ of $B_R(T)$ the diameter is either the radius,
or the chord, depending on the central angle (if it exceeds $\pi/3$):
so we obtain $d^{*}:=\diam K^{*} \le \diam S^{*} = \max
(R,2R\sin(\beta/2))$. Recall the definition of $R$ as three times the
maximum of the two sides from $T$ of the triangle
$\triangle(T\ze\ze')$: calculating from the sine theorem we thus obtain
$R\le 3s/\sin \beta$, and, moreover, in case $\beta >\pi/2$ we even get
$R \le 3 s$ (as then the side $\overline{\ze\ze'}$, opposite to the
largest angle $\beta>\pi/2$, is necessarily the longest side of the
triangle). On combining these estimates, we finally get
$$
d^{*} \le
\begin{cases}
\frac{3}{\sin \beta} s \qquad  \qquad &{\rm if}~~ \beta \le \pi/3 \\
\frac{3}{\cos(\beta/2)} s \le 3\sqrt{2} s &{\rm if}~~ \pi/3 \le \beta \le \pi/2 \\
3 \max\left(1,2\sin(\beta/2) \right) s \le 6s &{\rm if}~~ \beta >\pi/2
\end{cases}
~, \qquad d^{*} \le \frac{6}{\min\left(1,2 \sin \beta \right)} s
$$
Now let us assume that $s\le s_0$. This means that $d^{*} \le \frac{6}{\min\left(1,2 \sin \beta \right)} \cdot \frac{\min(1,2\sin\beta)}{384} d = d/64$.

Consider now the case when the number $\mu$ of zeroes of $p$ in $K^{*}$ is $ \ge n/2$: we claim that then (i) of the stated alternative holds true.

The condition $\mu\ge\dfrac{n}{2}$ can be written with $k=64$ as $\mu\ge \dfrac{3\log 2}{\log k} n$. Therefore, an application of Lemma \ref{l:transfdiam} with $k=64$ provides that we necessarily have $|p(\ze)|$ and $|p(\ze')|$ rather small, smaller than $2^{-n}\|p\|_\infty$, proving the first part of the claim in (i). (In fact, in this case we also found that the same must hold throughout all of $K^{*}$, so in particular on all points of the arc $\widetilde{\ze\ze'}$ between $\ze$ and $\ze'$.)

As for the second part of (i), we certainly have $\ze,\ze' \not\in \HH$ whenever $cn^{-2/q} \ge 2^{-n}$ with the constant $c$ defined in \eqref{eq:Hsetdef}; reformulating, it suffices to have $\frac12 \left(\frac{1}{n^2 8 \pi (q+1)}\right)^{1/q} \ge 2^{-n}$. As $q\ge 1$ and the left hand side is easily seen to increase in function of $q\ge 1$, it suffices to show this for $q=1$; and for $q=1$ the inequality becomes $2^n/n^2 \ge 32 \pi$, which holds for $n\ge 15$, as for $n\ge 15$ the left hand side is an increasing function of $n$ and its value at $n=15$ is $2^{15}/15^2 > 2^{15}/16^2= 32\cdot 4 > 32\pi$. Thus (i) is satisfied concluding the proof in this case.

In the other case (when $\mu <n/2$), however, we must have $\nu \ge n/2$. Therefore, in this case \eqref{twosidestogether} furnishes (ii) of the stated alternative and the proof concludes also in this case.
\end{proof}

\section{Proof of Theorem \ref{th:nlogn}}\label{sec:nlogn-mainproof}

In this section we prove the main result of the paper, that is Theorem \ref{th:nlogn}.
More precisely, we also get the following  explicit estimate of $M_{n,q}$ for large values of $n$.
\begin{theoremone}\label{th:nlogn2} Let
$K\Subset \CC$ be any compact convex domain.
Then for any $q\ge 1$, $n\ge n_0(K)=\max\left(10^{20}, d^5/w^5\right)$
and all $p\in \PK$, we have
\begin{equation*}\label{genrootineq}
\Norm{p'}_{q}\ge \frac{1}{240~000} ~\frac{w^2}{d^3} \frac{n}{\log n}  \Norm{p}_{q} \qquad {\rm i.e.} \qquad M_{n,q} \ge \frac{1}{240~000} ~\frac{w^2}{d^3} \frac{n}{\log n}~.
\end{equation*}
\end{theoremone}

\begin{proof}
The proof is divided into five parts. In the first  part~\ref{subsec1}, we  introduce the set $\GG$ of ``good points'', for which the tilted normal estimate \eqref{eq:oldyield1} may be applied. In part ~\ref{sec:elementaryarcs} we construct and describe a set $\LL$, which covers $\Gamma \setminus \GG$. The integral of $|p'|^q$ over $\HH$ is estimated in parts~\ref{sec:CaseI}, \ref{sec:CaseII}. These parts are computationally quite expansive.

\subsection{The subset $\boldsymbol{\GG}$ of ``good points''}\label{subsec1}
As before in Lemma \ref{l:fromoldproof}, we fix also here the angle $\theta$ as $\theta:=\arcsin(w/d)/80$, so that the angle $\ff:=\pi/2-2\theta$ will satisfy
$2\pi/5 < \ff <\pi/2$. Further, we will fix a parameter $r>0$. The value of this will be of the order $\log n/n$, so very small for $n$ large. Later in part \ref{sec:CaseI} the ``tilted normal estimate" of Lemma \ref{l:fromoldproof} will be applied with chord lengths at least $r$.

At the outset we will assume $r<w/108$; later we will need more restrictions to $r$, but $r<w/108$ will certainly be satisfied. In all, our condition on $r$ will be expressed as $0<r\le r_0=r_0(w,d)$, so that $r_0$ depends only on the parameters $d=d_K$ and $w=w_K$, but not on anything else, in particular not on the degree $n$. On the other hand we will also assume that $n\ge n_0:=n_0(w,d)$, again depending only on $w$ and $d$, and on nothing else. The bound $r_0$ will be specified later;
$n_0$ will be equal to the $n_0(K)$, already set in the assertion of Theorem \ref{th:nlogn2}.

At any point $\ze \in \DK=\Gamma$ one can consider all normals, and the respective tangents; and the tilted normal lines with angles $\ff$ from
the tangent \emph{measured from halflines of the tangent lines}, i.e., with $\pm 2\theta=\pm (\pi/2-\ff)$ from the (inner) normal directions.

If any of these tilted lines intersect $K$ in a sufficiently long chord, i.e., in a chord at least as long as $r$, or if one of the tilted lines does not intersect the interior of $K$ at all, then we will apply the ``tilted normal estimate" \eqref{eq:oldyield1} of Lemma \ref{l:fromoldproof}. Later in part \ref{sec:CaseI} we will see how for these points an application of Lemma \ref{l:fromoldproof} suffices. So these points we will call ``good points", the full set of good points being $\GG \subset \Gamma$. Our main concern will be to deal with points in $\Gamma \setminus \GG$.

\subsection{The  subsets $\boldsymbol{\FF}$ and $\boldsymbol{\LL}$ of $\boldsymbol{\DK}$}\label{sec:elementaryarcs}

\subsubsection{The family $\I$ of ``elementary small arcs''.}\label{subsec2} Once a point $\ze \not\in \GG$, it means that taking any (outer) normal direction $\bnu=-e^{i\sigma}$ to $K$ at $\ze$, at least one of the two chords $(\ze+e^{i(\sigma \pm 2\theta)}\RR)\cap K$ is short -- shorter than $r$. As stated in Claim~\ref{cl:anglediamchord} and Claim~\ref{cl:plusminus}, the ``small part of $\Gamma$" and the respective small part of $K$, encircled by this part of the boundary, is always proportional to the chord length $\de=\de(\ze, \sigma \pm 2\theta)=|(\ze+e^{i(\sigma \pm 2\theta)}\RR)\cap K|$, whence is small, too.

Together with a $\ze \in \Gamma\setminus \GG$, there is thus a chord line $\ell$ of direction $\pm 2\theta$ from the inner normal  direction $e^{i\sigma}$ with $\ell \cap K= [\ze, D]$ with $D:=D(\ze, \sigma \pm 2\theta ) \in \Gamma$ and $|D-\ze|=\de(\ze, \sigma \pm 2\theta)<r$. Consider the smaller arc of $\Gamma$ bounded by these two points $\ze$ and $D$ (endpoints included). Such arcs will be called \emph{elementary small arcs} -- thus a point $\ze\not \in \GG$ if and only if it defines such an elementary small arc
$I$ to which it is one of the endpoints and so in particular $\ze \in I$. The family of all such elementary small subarcs will be denoted by $\I$ and their union is by $\EE$, so that $\EE:=\cup_{I\in\I} I$. We clearly have $(\Gamma\setminus\GG) \subset \EE$. Note that the name ``elementary \emph{small} arc" is well-justified
 because any elementary small arc $I \subset \Gamma$ is of length not exceeding $4dr/w$ in view of Claim~\ref{cl:anglediamchord}, (iv).

However small these arcs $I\in\I$ are, they exhibit a certain largeness, too. Namely, along any elementary small subarc $I$ the total variation $\Var[\al,I]$ of $\alpha$ is at least $\ff$. To see this, assume that $I=\widetilde{\ze D}:=\{P\in \Gamma~:~ \ze\prec P \prec D\}$, say, with the chord $[\ze,D]=K\cap(\ze+e^{i(\sigma-2\theta)})\RR$ small. (The other case when $I=\widetilde{D\ze}$, i.e. $D \prec \ze$ and $[\ze,D]=K\cap(\ze+e^{i(\sigma+2\theta)})\RR$ small is entirely symmetrical again.) So the (positively oriented) directional angle of (one) tangent $t$ at $\ze$ is $\sigma-\pi/2$, the chord $\vec{\ze D}$ is of direction $\sigma-2\theta$, and by convexity of $\Gamma$ any tangent $t'$ at $D$ must have direction $\al'\ge \sigma-2\theta$. Thus we indeed find $\Var[\al,I]=\al_{+}(D)-\al_{-}(\ze) \ge \al'-(\sigma-\pi/2) \ge (\sigma-2\theta)-(\sigma-\pi/2)=\ff$.

As an immediate result, there can be \emph{at most four disjoint such subarcs} in $\Gamma$: for already along five disjoint elementary subarcs the total variation would be at least $5 \cdot \ff > 2\pi =\Var[\alpha,\Gamma]$, a contradiction. So let us chose, once for all, a \emph{maximal family} of such \emph{disjoint} elementary small subarcs $\Lj$ ($j=1,\dots,k$), $k\le 4$.
We of course have then only $\FF:=\cup_{j=1}^k \Lj \subset \EE$, and cannot state that even $\FF$ covers $\Gamma\setminus \GG$,
but on the other hand we know that any point $z\in\EE$ belongs to some elementary arc $I\in\I$, which intersects some of these $\Lj$. As we have $|I|\le 4rd/w$, always, it means that a point $z\in\EE$ cannot be farther (measured in arc length along $\Gamma$) from $\FF$ than $4rd/w$.

\subsubsection{A covering $\LL$ of the sets $\FF$ and $\EE$.}\label{subsec3}
In view of the above, extending each $\Lj$ along $\Gamma$ \emph{in both directions} by $4dr/w$ in arc length, we obtain a subset
$$
\LL:=\{ z\in \Gamma~:~ \dist(z,\FF) \le 4rd/w \} \subset \Gamma,
$$ which will contain all points of $\EE$, and thus also cover $\Gamma\setminus \GG$ again. So we find $\GG \cup \LL =\Gamma$. Let us also record right here that the total arc length measure of the so constructed set $\LL$ is $|\LL|\le 4 \cdot 12 rd/w = 48 rd/w$.

So there are points in $\Gamma\setminus \LL$, and fixing one such point $C\in \Gamma \setminus \LL$, we can start parametrization of $\Gamma$ from that point: it means that $\gamma:[0,L]\to \Gamma$ will define a unique ordering of points of $\Gamma \setminus \{C\}$, so in particular of points of $\LL$. In this ordering let us write $\Lj=\arc{P_j P_j'}$; so it can be the case that  $P_j'=D(P_j)$, but also that $P_j=D(P_j')$, depending on the initial point of small chord length in the construction of the arc.

So, $\LL:=\cup_{j=1}^k \arc{Q_j Q_j'}$, where $Q_j \prec P_j \prec P_j' \prec Q_j'$, and the arc length measures are $|\arc{Q_j P_j}|=4rd/w$, $|\arc{P_j P_j'}| \le 4rd/w$, $|\arc{P_j' Q_j'}|=4rd/w$, and altogether $|\arc{Q_jQ_j'}| \le 12rd/w$.

There is only a slight technicality here: these vicinities $\arc{Q_j Q_j'}$ of the disjoint small elementary arcs $\Lj=\arc{P_j P_j'}$ need not remain disjoint. However, \emph{no three of them may chain together}. Indeed, assume this to happen: that would result in a subarc $\Gamma'$ of $\Gamma$, altogether not longer than $3 \cdot 12 rd/w = 36rd/w$, with a total variation of the tangent direction already exceeding $ 3 \cdot \ff > 6\pi/5$. This is, however, impossible. Indeed, then the tangents $t$ and $t'$ at the endpoints $\ze \prec \ze'$ of $\Gamma'$ would intersect at a point $T$ on the other side of $\Gamma'$, and the triangle $\triangle=\con(T,\ze,\ze')$ would contain $\Gamma \setminus \Gamma'$: and then using
$\diam (\Gamma \setminus \Gamma') \le \diam(\triangle) \le \dfrac{1}{\sin\pi/5} |\ze'-\ze|< 2|\ze'-\ze|$ we would get
\begin{align*}
\diam(K) & =\diam(\Gamma) \le
 \diam(\Gamma \setminus \Gamma') +\diam(\Gamma')  < 2 |\ze'-\ze| + \diam(\Gamma')
\\ & \le 2 |\Gamma'| + |\Gamma'| ~\le ~3 \cdot 36rd/w ~=~ 108 rd/w < d,
\end{align*}
a contradiction.

In view of the above, $\LL=\cup_{m=1}^{k_0} \A_m$, where $\A_m$ are to denote the connected components of $\LL$, their number is $k_0 \le k\le 4$, and each of the components have arc length exceeding $8rd/w$. More precisely, each of the connected components (arcs) of $\LL$ consists of some (one or two) of the prefixed disjoint elementary arcs $I_j=\arc{P_jP_j'}$ -- among which we can now chose one arbitrarily, if there are two -- and also some part preceding, and some other part following this selected elementary arc. We will thus write for one arbitrary connected component $\A_m$ of $\LL$ that $\A_m=\arc{Q_m Q_m'}$ with $Q_m \prec P_{j(m)} \prec P_{j(m)}' \prec Q_j'$ with the parts $\A_{m-}:=\arc{Q_m P_{j(m)}}$ and $\A_{m+}:=\arc{P'_{j(m)} Q'_m}$ having arc length measure at least $4rd/w$ and at most $16rd/w$, and the intermediate (``central") part $I_{j(m)}=\arc{P_{j(m)} P_{j(m)}'}$ at most $4rd/w$; and in all,
\begin{equation}\label{mA}
8rd/w < |\A_m|\le 24rd/w    \quad (m=1,\dots,k_0).
\end{equation}

Let us briefly summarize our construction of subsets of $\DK$. We started with points $\zeta\not\in \GG$, considered the elementary short subarcs $I=\arc{\ze D}$ or $\arc{D \ze}$ generated by any such $\ze$, and took the union $\EE:=\cup_{I\in\I} I$ of these subarcs, obviously covering $\Gamma\setminus\GG$. Next, we selected a maximal disjoint subset of elementary subarcs and their union $\FF:=\cup_{j=1}^k I_j$, which is only a subset of $\EE$; but then took a proper neighborhood $\LL$ of $\FF$ to cover $\EE$, and whence also $\Gamma\setminus \GG$ again. The advantage of these steps back and forth are that the resulting set $\LL$ not only covers $\EE \supset (\Gamma\setminus \GG)$, but it also has a manageable structure: it consists of $k_0\le k \le 4$ connected subarcs $\A_m$ of $\Gamma$, all of which having arc length measure between $8rw/d$ and $24rd/w$, and each of which is easily divided into three parts: one selected elementary small subarc $I_{j(m)}$ from the disjoint system $\{I_{j}\}_{j=1}^k$ as ``central part", and the preceding and following parts $\A_{m-}, \A_{m+}$, both of the size of order $rd/w$, too. It is important that the ``central part" exhibits a change of tangent angle function at least $\ff$, and its arc length is bounded by that of the surrounding parts (i.e. of the order $rd/w$). One such connected subarc $\A:=\A_m$ is depicted in Figure 2.

\begin{figure}
\begin{center}
 \includegraphics[width=10cm,keepaspectratio]{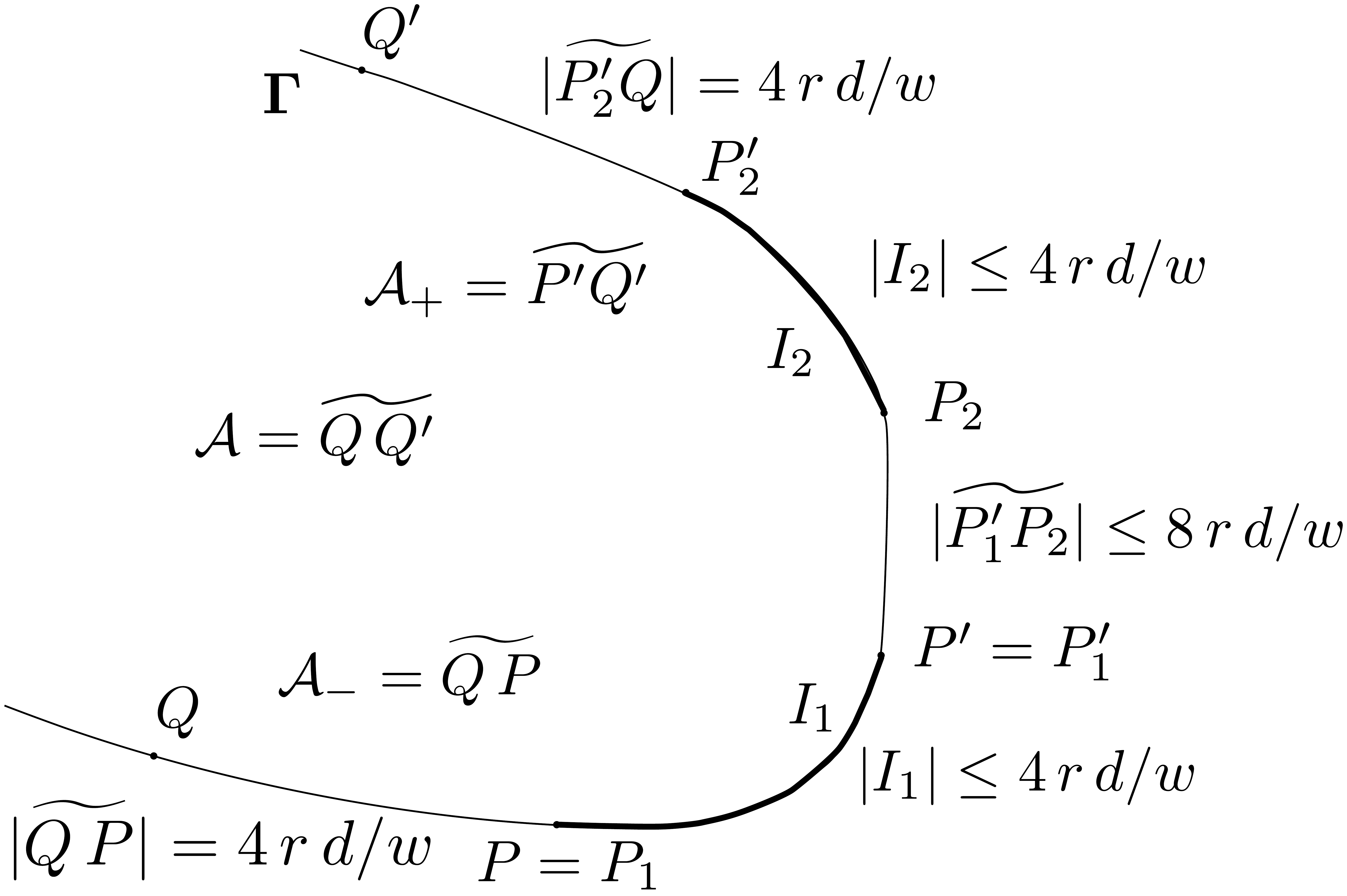}\\
  \caption{One connected component $\A$ of the set $\LL$.}\label{Fig-A}
\end{center}
\end{figure}

We have already seen in Lemma \ref{l:Hlogp} that $\displaystyle \int_\HH |p|^q |dz| \ge \frac12 \int_\Gamma |p|^q |dz|$.
In the rest of the proof, we distinguish two cases:
\begin{align*} &\textrm{Case I.} \quad \int_{\HH \cap \LL} |p|^q |dz| \le \frac 12 \int_\HH |p|^q |dz|; \vphantom{\int\limits_0}
\\ &\textrm{Case II.}  \quad  \int_{\HH \cap \LL} |p|^q |dz| > \frac 12 \int_\HH |p|^q |dz|.
\end{align*}

\subsection{Case I}\label{sec:CaseI}

Note that in this case we have
 $$\displaystyle\int_{\HH \setminus \LL} |p|^q |dz| \ge \frac12 \int_{\HH } |p|^q |dz| \ge \frac 14 \int_\Gamma |p|^q |dz|.
 $$
So let then $\ze \in \HH\setminus \LL$ be any point. As $\ze \not \in \LL$, it follows that $\ze\in\GG$, that is its ``tilted normal" chord length has $\delta(\ze, \sigma\pm2\theta)\ge r$. So, from \eqref{eq:oldyield1} of  Lemma \ref{l:fromoldproof} we get  the estimate
\begin{equation}\label{eq:tiltednormalest}
|p'(\ze)| \ge \left( 0.001 \frac{w}{d^2} n - \frac{3}{20 r}  \log n   \right) |p(\ze)| \quad (n \ge 73).
\end{equation}
Therefore, if
\begin{equation}\label{rnbounds}
r := r(n):=300 \dfrac{d^2}{w} \dfrac{\log n}{n},
\end{equation}
 then $|p'(\ze)| \ge 0.0005 \dfrac{w}{d^2} n |p(\ze)|$, and we get
$$
\int_{\HH\setminus \LL} |p'(\ze)|^q  |d\ze| \ge \left( 0.0005 \frac{w}{d^2} n\right)^q  \int_{\HH \setminus \LL} |p(\ze)|^q |d\ze| \ge \left( 0.0005 \frac{w}{d^2} n\right)^q   \frac14 \int_{\Gamma} |p(\ze)|^q |d\ze|.
$$
It follows that in this case
\begin{equation}\label{eq:firstcaseconclusion}
\|p'\|_q \ge \left( \int_{\HH\setminus \LL} |p'|^q  \right)^{1/q} \ge \frac{1}{4^{1/q}} 0.0005 \frac{w}{d^2} n \left( \int_{\Gamma} |p|^q \right)^{1/q} > 0.0001 \frac{w}{d^2} n \|p\|_q^q,
\end{equation}
which closes the argument for all $n$ sufficiently large (so that $r(n)<r_0$ holds).
Note that in this case we obtained a constant times $n$ oscillation, not only of order $\log n/n$.

\subsection{Case II} \label{sec:CaseII} In the remaining other case we have that
$$
\displaystyle \int_{\HH \cap \LL} |p|^q |dz| > \frac 12 \int_\HH |p|^q |dz|, \ \textrm{therefore,} \
\int_\LL |p|^q \ge \int_{\LL\cap \HH} |p|^q \ge \frac12 \int_{\HH} |p|^q \ge \frac14 \int_{\Gamma} |p|^q.
$$
Recall that $\LL$ consists of $k_0\le 4$ arcs, each of length between $8rd/w$ and $24 rd/w$. Let us select one arc $\A_m=\arc{Q_mQ_m'}$, where $\displaystyle\int_{\A_m} |p|^q$ is \emph{maximal} among these at most four arcs. To relax notation, from now on let us drop the indices $m$ and $m(j)$ in the following and write $\A$ for $\A_m$, $P$ for $P_{j(m)}$, $Q$ for $Q_m$ etc. So as it was said before, we fix one elementary small subarc $I=I_{j(m)}=\arc{PP'} \subset \A$ as the ``central part" of $\A$ and we write $\AM:=\arc{QP}$ and $\AP:=\arc{P'Q'}$ for the parts preceding resp. following it.

By construction, we necessarily have
\begin{equation}\label{intA}
\displaystyle \int_\A |p|^q \ge \frac14 \int_\LL |p|^q \ge  \frac1{16} \int_{\Gamma} |p|^q.
\end{equation}
We also put
$$
u:=\min_\A |p(z)| \quad \textrm{and} \quad v:=\max_\A |p(z)|=\|p\|_{L^\infty(\A)}.
$$
With these quantities, we consider two subcases next:
\begin{align*} &\textrm{Subcase II.1.} \qquad 2u < v;
\\ &\textrm{Subcase II.2.} \qquad 0<u \le v \le 2u .
\end{align*}

\subsubsection{Subcase II.1: $2u < v$}
We estimate the integrals using the H\"older inequality and the trivial estimation of the variation of $p$ on $\A$ as follows:
\begin{align}\label{eq:variationLjr}
|\A|^{1-1/q} \left( \int_{\A} |p'|^q \right)^{1/q} & \ge \int_{\A} |p'| \ge |v-u| \ge \frac{v}{2} = \frac12  \|p\|_{L^\infty(\A)}  \notag
\\ & \ge \frac12 \left( \frac{1}{|\A|} \int_{\A} |p|^q \right)^{1/q} = \frac12 |\A|^{-1/q} \left(\int_{\A} |p|^q \right)^{1/q},
\end{align}
i.e., we obtain
$\displaystyle 2 |\A| \left(  \int_{\A} |p'|^q \right)^{1/q}   \ge \left( \int_{\A} |p|^q \right)^{1/q}  \ge \left( \frac1{16} \int_{\Gamma} |p|^q \right)^{1/q} $, and so with $r=r(n)$
\begin{equation}\label{2ulev}
\begin{aligned}
\|p'\|_q &   \ge \left( \int_{\A} |p'|^q  \right)^{1/q}  \ge \frac{1}{2 |\A|} \left( \frac1{16} \int_{\Gamma} |p|^q \right)^{1/q}
\ge \frac{\|p\|_q}{16^{1/q} \cdot 2 \cdot 24 \frac{rd}{w}}
\\ & > \frac{w}{800 ~d ~ r(n)} \|p\|_q =
  \frac{w}{800 ~d \cdot 300 \frac{d^2}{w} \frac{\log n}{n}}   \|p\|_q
  =  \frac{1}{240~000} ~\frac{w^2}{d^3} ~ \frac{n}{\log n} ~\|p\|_q.
\end{aligned}
\end{equation}

\subsubsection{Subcase II.2:  $0<u \le v \le 2u$}
We will consider any two points $\ze \in \AM$ and $\ze' \in \AP$. Note that between these two points there lies the elementary small subarc $\Lj$, whence if $t$ and $t'$ are tangent lines at $\ze$ resp. $\ze'$ to $K$, with directional angles $\alpha$ and $\alpha'$, respectively, then we necessarily have $\alpha'-\alpha \ge \ff$.

This is the place where we need Lemma \ref{l:twosides}.
The distance of the points is $s:=|\ze-\ze'| \le |\A| \le 24 \frac{rd}{w}$, which must not exceed $s_0$ of the condition
 of Lemma \ref{l:twosides}. For the angle $\beta:=\pi-(\al'-\al)$ we already know by construction
 that $\beta \le \pi -\ff \le 3\pi/5$, so $\sin \beta \ge 1/2$ unless $\beta < \pi/6$.
Therefore, we need to care for small $\beta$ only.
 However, according to Claim~\ref{cl:anglediamchord} from  Section~\ref{sec:geom},
we also know that $\beta > \arcsin ((w-s)/d)$, whence $\sin\beta \ge (w-s)/d$ whenever $0<\beta <\pi/2$.
So altogether we find that $\sin \beta \ge \min(1/2, (w-s)/d) \ge w/(2d)$ if we assume also $s\le w/2$.

At this point we need to specify a sufficient condition in terms of $r$ for the chord length $s:=|\ze'-\ze|$
to stay below $\min(s_0,w/2)$: it suffices if
$$
\dfrac{24rd}{w} \le \dfrac{w}{384} ~\left( \le \dfrac{\min(1,2\sin\beta)}{384} d \right),
$$
that is, if we have $r \le \dfrac{1}{24\cdot 384} \dfrac{w^2}{d}$ so e.g. if $r \le r_1:=10^{-4} \dfrac{w^2}{d}$ (which is much smaller than the initial condition $w/108$ was).

The alternative of the said Lemma \ref{l:twosides} has (i) with rather small values of the polynomial $p$. However, the variation of the values all over $\A$ remains within a factor $2$ in our case. Thus we conclude that even for the maximum we must have $v:=\|p\|_{L^\infty(\A)} \le 2^{1-n} \|p\|_{L^\infty(K)}$. So from the above
$$
\|p\|_{L^q(\DK)}^q = \int_\Gamma |p|^q  \le 16 \int_\A |p|^q \le 16 |\A|  v^q \le 16 \cdot 24 \frac{d}{w} \cdot 300 \frac{d^2}{w} \frac{\log n}{n} \cdot 2^{q-qn} \|p\|_{L^\infty(K)}^q.
$$
Next we show that this is not possible. And indeed, according to the Nikolskii type estimate of Lemma \ref{l:Nikolskii}, we must have $\|p\|_{L^q(K)} \ge \left( \frac{d}{2(q+1)}\right)^{1/q}
~\|p\|_{L^\infty(K)} ~ n^{-2/q}$, so combining with the latter formula we get
$$
\|p\|_{L^q(\DK)}^q \le 16 \cdot 24 \frac{d}{w} \cdot 300 \frac{d^2}{w} \frac{\log n}{n} \cdot 2^{q-qn}
\cdot \frac{2(q+1)}{d} ~\|p\|_{L^q(\DK)}^q ~ n^2,
$$
that is
$$
\frac{2^{qn}}{n \log n} \le 16 \cdot 24 \cdot 300 \cdot \frac{d^2}{w^2}  \cdot 2^{q+1} (q+1)
$$
which clearly fails for $n$ large enough.

To be more precise, consider any fixed $n\ge 73$ and the function $2^{(n-1)q}/(q+1)$: then this is clearly an increasing function of
$q \in [1,\infty)$, so it suffices to establish a contradiction with $q=1$; and so it suffices to demonstrate a contradiction with
$$
\frac{2^{n}}{n \log n} \le 8 \cdot 16 \cdot 24 \cdot 300  \cdot \frac{d^2}{w^2} = 900 \cdot 2^{10} \cdot \frac{d^2}{w^2} = ~ 921~ 600  \cdot \frac{d^2}{w^2}
$$
for sufficiently large $n$. As for the left hand side, we can write $\dfrac{2^x}{x \log x} > 2^{x/2} \cdot f(x)$ with $f(x):= \dfrac{2^{x/2}}{x \log x}$, and the latter is an increasing function of the variable $x$ for
$x\ge 73$, so we have $f(x)\ge f(73)>310\,290\,286$, $x\ge 73$. Therefore, the desired contradiction will arise if $2^{n/2}f(73)> 921\,600 \, d^2/w^2$, in particular, if $2^{n/2}>  d^2/w^2$. Actually, it suffices then to take
\begin{equation*}\label{n1}
n\ge n_1:=\max(73, 6\log(d/w)).
\end{equation*}

\bigskip
It remains to consider the second alternative (ii) of Lemma \ref{l:twosides}. As is clarified above, we already know that the occurring angle $\beta$ satisfies $\arcsin (w-s)/d \le \beta \le 3 \pi/5$, so this alternative of the assertion of Lemma \ref{l:twosides} works with $\sin \beta \ge \dfrac{w}{2d}$ for sure. That is, we have
$$
\left| \frac{p'}{p}(\ze) \right| + \left| \frac{p'}{p}(\ze') \right| \ge \frac{3\sin\beta}{8d} n \ge \frac{3}{16} ~ \frac{w}{d^2} ~n
\qquad \left( \forall \ze \in \AM \quad \textrm{and} \quad \forall \ze' \in \AP \right).
$$
In particular, if there is any point $\ze \in \AM$ with $\left| \dfrac{p'}{p}(\ze) \right| \le \dfrac{3}{32} ~ \dfrac{w}{d^2} ~n$,
then we must have $\left| \dfrac{p'}{p}(\ze') \right| \ge \dfrac{3}{32} ~ \dfrac{w}{d^2} ~n$ on \emph{all over} $\AP$, and, conversely,
if there is such a ``small value point" on $\AP$, then we must have this lower estimation for all over $\AM$. So, either both subarcs $\AM, \AP$
satisfy this lower estimation, or at least one of them must satisfy it at all of its points. So assume, as we may, that $\AP$ satisfies this lower estimation: this yields
$$
\int_{\AP} |p'(\ze')|^q |d\ze'| \ge \left(\frac{3}{32} ~ \frac{w}{d^2} ~n\right)^q ~\int_{\AP} |p(\ze')|^q |d\ze'|.
$$
Recall that $|\AP| \ge 4rd/w$, while $|\A|\le 24 rd/w$, and that $u\le | p(z)| \le v \le 2u$ holds all over $\A$. These furnish
$$
\int_{\AP} |p(\ze')|^q |d\ze'| \ge  |\AP| u^q \ge \frac{4rd}{w} 2^{-q} v^q \ge \frac{|\A|}{8} 2^{-q} v^q \ge 2^{-q-3} \int_\A |p|^q \ge 2^{-q-7} \int_\Gamma |p|^q,
$$
using also~\eqref{intA}, established at the beginning of the case in consideration.

So in all, we are led to
$$
\|p'\|_q^q \ge \int_{\AP} |p'|^q \ge \left(\frac{3}{32} ~ \frac{w}{d^2} ~n\right)^q ~\int_{\AP} |p|^q \ge
\left(\frac{3}{64} ~ \frac{w}{d^2} ~n\right)^q \frac{1}{2^7}  \|p\|_q^q.
$$
So in this case, we arrive at
\begin{equation}\label{vle2u}
\|p'\|_q \ge \frac{3}{64 \cdot 2^{7/q}} ~ \frac{w}{d^2} ~n \|p\|_q > 0.0003 \frac{w}{d^2} ~n \|p\|_q.
\end{equation}

\bigskip
Collecting the above estimates \eqref{vle2u},
\eqref{eq:firstcaseconclusion}, \eqref{2ulev}  we arrive at
$$
\|p'\|_q \ge \min\left( 0.0003 \frac{w}{d^2} ~n,\, 10^{-4} \frac{w}{d^2} ~n, \,  \frac{1}{240~000} ~\frac{w^2}{d^3} ~ \frac{n}{\log n} \right) \|p\|_q = \frac{1}{240~000} ~\frac{w^2}{d^3} ~ \frac{n}{\log n} ~\|p\|_q
$$
provided that all our conditions are met:
$r=r(n) \le r_1:=10^{-4} \dfrac{w^2}{d}$ and
 $n \ge n_1:=\max(73, 6\log(d/w))$.
Note that here $r(n):= 300 \dfrac{d^2}{w} \dfrac{\log n}{n}$
 depends on $n$, decreases to $0$,
and it suffices to find an index $n_0\ge n_1$ such that at $n=n_0$, and whence for all $n\ge n_0$, too, the inequality $r(n)\le r_1$ holds. That is, we want
$$
300 \dfrac{d^2}{w} \dfrac{\log n}{n} \le 10^{-4} \dfrac{w^2}{d}
\qquad\textrm{or}~\textrm{equivalently}\qquad
\frac{n}{\log n} \ge 3 \cdot 10^{6}  ~ \frac{d^3}{w^3}.
$$
For example if
 $n_0:= \max\left(10^{20}, \dfrac{d^5}{w^5}\right)$
and $n\ge n_0$,
 then we certainly have $\dfrac{n}{\log n} \ge n^{0.93} =
 n^{0.33} \cdot n^{0.6} > 10^{6.5} \dfrac{d^3}{w^3} >
3 \cdot 10^6 \dfrac{d^3}{w^3}$.
Therefore, for any $n \ge n_0$ -- which is much larger than the previously found bound $n_1$ -- the required $r(n)<r_1$, whence all the above arguments hold true and Theorem~\ref{th:nlogn2} follows.
\end{proof}

\begin{remark}
Let us  note  that for any fixed $n\in \NN,$ the set $\mathcal{S}_n=\{p\in \PP_n: \|p\|_q=1\}$ is compact in $L^q(\partial K).$
Hence $M_{n,q}=\inf_{p\in \mathcal{S}_n} M_q(p)>0$ for any $q\in[1,+\infty],$ as $M_q(p)>0,$ $p\in \mathcal{S}_n$.
The quantity  $\|p\|_q$ is continuous with respect to $q\in[1, +\infty]$ (see, e.g., \cite[6.11]{HLP}),  therefore $M_{n,q}$
 is continuous   too. Since $M_{n,\infty} >0$, we conclude that $M_n:=\inf_{q\in[1,+\infty]} M_{n,q}>0,$
and we can certainly take
$$
c_K=\min\left(M_2 \frac{\log 2}{2},\, \ldots, M_{n_0}\frac{\log n_0}{n_0},  \, \frac{1}{240~000} ~\frac{w^2}{d^3} \right).
$$
Consequently, for all $q\ge1,$ $n\ge 2$ and $p\in \PP_n(K)$
$$
\Norm{p'}_{q}\ge c_K \frac{n}{\log n}  \Norm{p}_{q}~,
$$
where the constant $c_K$ depends only on the set $K$.
\end{remark}

\section{Concluding remarks}\label{sec:conclusionn}

The proof of our main result shows that we can even reach $cn$ order of oscillation, and even pointwise estimates, apart from the set of critical small elementary arcs, where the intersection of the (tilted) normal with $K$ is very small, smaller than $c\log n/n$, but not zero (as in case of $\ze=D$ we still have an order $n$ lower estimate at $\ze$, see Lemma \ref{l:fromoldproof} (ii)). However, when $n\to \infty$, the quantity $\log n/n$ tends to 0, and in fact we see that for most domains the set of critical elementary small arcs becomes empty for $n$ large.

More precisely, one can do the following. Take $\EE:=\EE(\ff,r)$ be the union of all (closed) elementary small arcs $\arc{\ze D}$,
defined in subsection \ref{subsec2}. Then obviously $\EE(\ff,r)$ is a decreasing set function of the parameter $r>0$; also it is clarified above that it consists of $k\le 4$ connected arc pieces of $\Gamma$. Each elementary small arc has a total variation of the tangent angle function at least $\ff$, so in each connected component the same holds. Moreover, as discussed in connection with the construction of $\LL$, any such connected component has to be at most of length $24rd/w$. Taking the limit when $r\to 0$ i.e. taking $\EE^{*}:=\cap_{r>0} \EE(\ff,r)$, we find that either $\EE^{*}=\emptyset$, or that $\EE^{*}$ consists of a few isolated points (at most 4), where the variation (jump) of the tangent angle function $\alpha$ reaches $\ff$. In case $\EE^{*}=\emptyset$, we can still conclude that for $n$ large enough (even if this largeness ineffectively depends on the geometry of the domain $K$) this critical part of the proof can be skipped and there holds a constant times $n$ oscillation estimate. This is not much different from the phenomenon described in \cite{PR}.
If, on the other hand, $\EE^{*}\ne \emptyset$, then we know that $\DK$ has vertices with (almost) right angle jumps of the tangents. With a suitable choice of $\ff$ we can thus prove a sharpening of the result of Theorem \ref{th:posdepth} in the extent that the dependence of the occurring constant is better, than the one in Theorem \ref{th:posdepth}.

Actually, we have the following direct corollary of the results of Section \ref{sec:tiltednormal}.

\begin{corollary} Assume that the compact convex domain $K$ does not have any boundary points where the jump
 of the tangent directional function $\alpha$ would reach $\ff=\pi/2-\arctan(w/d)/40$.
 Then we have $\EE(\ff,r)=\emptyset$ for small enough $r$, and, as a result, $\left|\dfrac{p'(\ze)}{p(\ze)}\right| \ge c \dfrac{w}{d^2} n$ at each boundary points  $\ze\in\HH$ for $n \ge n_0(K)$.
Furthermore, then we also have $\|p'\|_q \ge c \dfrac{w}{d^2} n \|p\|_q$ for $n \ge n_0(K)$.
\end{corollary}

The above suggests that the truth in general could be as large as $c \dfrac{w}{d^2} n$, with an absolute constant $c$ -- the same order of magnitude as was found for the $\|\cdot\|_\infty$ case in \cite{Rev2}.
It is easy to obtain polynomials with as small an oscillation as $C/d n$ (see Theorem~\ref{th:orderupper}),
but recently Goryacheva in her master's thesis has worked out a construction with an even smaller oscillation: according to her work, an oscillation of order $C w/d^2 ~n$ is possible. Based on these observations and the maximum norm case, there seems to be enough evidence to further sharpen our Conjecture \ref{conj:cn}.

\begin{conjecture}\label{conj:cnprecise} There exists an absolute constant $c>0$ such that for all compact convex domains $K\Subset \CC$ and for any $p\in\PK$ we have $\|p'\|_{L^q(\DK)} \ge c \dfrac{w}{d^2} n \|p\|_{L^q(\DK)}$.
\end{conjecture}

Finally, let us analyze the strength of the arguments of the paper. Clearly our considerations are more involved than the ones used in \cite{Rev2} to derive Theorem \ref{th:convexdomain}, but from the end result neither this (for $q=\infty$), nor the sharper special cases Theorem \ref{th:posdepth} or \ref{th:ErodType} (for $1\le q<\infty$) follow. However, from one of the key elements, namely from Lemma \ref{l:fromoldproof}, a numerical improvement of Theorem \ref{th:convexdomain} follows.
\begin{corollary} Let $K\subset \CC$ be any compact convex domain. Then for all  $p\in \PK$ we have
\begin{equation}\label{genrootineq-improved}
\Norm{p'}_K\ge 0.001 \frac{w_K}{d_K^2} n  \Norm{p}_K~.
\end{equation}
\end{corollary}
\begin{proof} Chose $\ze\in\DK$ with $\Norm{p}_K=|p(\ze)|$,
draw any tangent and apply Lemma~\ref{l:fromoldproof}.
If we are in case (i), then an even better result is obtained.
If on the other hand we have some positive $\de_{\pm}$
then the actual value of $\de_{\pm}$ becomes irrelevant as
$\log\dfrac{\Norm{p}_K}{|p(\ze)|}=\log 1=0$.
\end{proof}

\end{document}